\documentclass[12pt, reqno]{amsart}
\usepackage{amscd, amssymb, latexsym, amsmath, amscd, enumerate}
\usepackage{amstext, anysize, amsthm}
\usepackage[sc]{mathpazo} 
\usepackage{enumitem}
\usepackage{ulem}
\usepackage{pst-node}
\marginsize{1.5cm}{1.5cm}{2cm}{2cm}
\usepackage{graphics}
\usepackage{color}
\usepackage[all]{xy}
\usepackage{tikz}
\usepackage{tikz-cd}
\usepackage[square, sort, numbers]{natbib}
\usepackage{hyperref}

\usepackage{mathrsfs}

\newtheorem{theorem}{Theorem}[section]

\newtheorem{corollary}[theorem] {Corollary}
\newtheorem{definition}[theorem]{Definition}

\newtheorem{lemma} [theorem]{Lemma}

\newtheorem{proposition}[theorem]{Proposition}
\newtheorem{remark}[theorem]{Remark}
\newtheorem{question}[theorem]{Question}
\numberwithin{equation}{section}

\newcommand{\F}{\mathbb{F}}
\newcommand{\Q}{\mathbb{Q}}
\newcommand{\C}{\mathbb{C}}

\newcommand{\GL}{{\rm GL}}

\newcommand{\Sp}{{\rm Sp}}
\newcommand{\Hom}{{\rm Hom}}

\newcommand{\Ind}{{\rm Ind}}

\newcommand{\tr}{{\rm tr}}

\newcommand{\Res}{{\rm Res}}

\newcommand{\Irr}{{\rm Irr}}
\newcommand{\Ps}{{\rm Ps}}

\newcommand{\diag}{{\rm diag}}
\newcommand{\Diag}{{\rm Diag}}
\newcommand{\Gr}{{\rm Gr}}

\newcommand{\cO}{\mathfrak{o}}

\newcommand{\Span}{{\rm Span}}

\title[On the Degenerate Whittaker space]{On the Degenerate Whittaker space for some induced representations of $\GL_4(\cO_2)$}  
\author{Ankita Parashar}
\address{Department of Mathematics, IIT Delhi, Hauz Khas, New Delhi - 110016}
\email{ankitaparashar216@gmail.com}

\author{Shiv Prakash Patel}
\address{Department of Mathematics, IIT Delhi, Hauz Khas, New Delhi - 110016} 
\email{shivprakashpatel@gmail.com}

\keywords{Degenerate Whittaker space,  Regular representations}
\subjclass[2020]{Primary: 20G25, Secondary: 20G05, 20C15}

\begin{document}

\begin{abstract}
Let $\cO_l$ be a finite principal ideal local ring of length $l$.
The degenerate Whittaker space  associated with a representation of $\GL_{2n}(\cO_l)$ is a representation of $\GL_n(\cO_l)$. 
For strongly cuspidal representations of $\GL_{2n}(\cO_l)$ the structure of degenerate Whittaker space 
is described by Prasad's conjecture, which has been proven for $\GL_4(\cO_2)$.
In this paper, we describe the degenerate Whittaker space 
for certain induced representations of $\GL_4(\cO_2)$, specifically those induced from subgroups analogous to the maximal parabolic subgroups of $\GL_4(\F_q)$.
\end{abstract}

\maketitle
%\tableofcontents

\section{Introduction}\label{Sec : 1}

Let $F$ be a finite unramified extension of $\Q_p$ or $\F_{p}((t))$. 
Let $\mathfrak{o}$ be its ring of integers with a uniformizer $\varpi$. 
For any positive integer $l$, let $\mathfrak{o}_l:=\mathfrak{o}/(\varpi^l)$, which is a principal ideal local ring of length $l$. 
Note that $\mathfrak{o}/(\varpi) \cong \mathbb{F}_q$, a finite field of order $q=p^f$ of characteristic $p>0$. 
Let $G = \GL_{2n}(\mathfrak{o}_l)$ and $P_{n,n} \subset G$ the subgroup of all $n \times n$ block upper triangular matrices in $G$.
Then $P_{n, n} =M \ltimes N$, where $M \cong \GL_n(\cO_l) \times \GL_n(\cO_l)$ is the set of  block diagonal matrices and $N \hookrightarrow G$ is given by $X \mapsto \left(\begin{matrix} I_n & X \\0 & I_n \end{matrix} \right)$, where $I_n$ denotes the $n \times n$ identity matrix. Note that $N \cong M_{n}(\cO_l)$.
%We fix  a character $\psi : \cO_l \rightarrow \C^{\times}$ such that $\psi|_{\varpi^{l-1} \cO_l}$ is non-trivial. 

Let $\psi : N \rightarrow \mathbb{C}^{\times}$ be a character of $N$.
For a representation $(\pi, V)$ of $G$, let $V_{N,\psi}$ be the maximal subspace of $V$ on which $N$ operates by $\psi$, i.e.
$$
V_{N,\psi} := \{ v \in V : \pi(X)v=\psi(X)v \,\, \text{ for~all }X \in N \}.
$$ 
This space is invariant under the action of  $M_{\psi} := \{ m \in M ~|~ \psi(mnm^{-1}) = \psi(n) \forall n \in N \}$, which is a subgroup of $M$. 
Thus we get a representation, say $(\pi_{N, \psi}, V_{N, \psi})$ of $M_{\psi}$, which is known as $(N, \psi)$-twisted Jacquet module of $\pi$ or $(N,\psi)$-Whittaker space of $\pi$. 
We now state the central question considered in this paper.
\begin{question} \label{main question}
For a given representation $\pi$ of $\GL_{2n}(\cO_l)$ and a character $\psi$ of $N$, what is $\pi_{N, \psi}$ as a representation of $M_{\psi}$?   
\end{question}

This question largely remains unanswered. 
For $l=1$, let $\psi : N \rightarrow \C^{\times}$ given by $\psi(X) = \psi_{0}(\tr(X))$, where $\psi_0$ is a non-trivial character of $\cO_1 \cong \F_{q}$. 
For a cuspidal representation $\pi$ of $\GL_{2n}(\F_q)$, a description of  $\pi_{N, \psi}$ is provided by Dipendra Prasad \cite{Prasad2000}.
His work has inspired further results, such as \cite{Himanshi2022, Himanshi, balasubramanian}, where variations of the character $\psi$ are considered.
Analogous questions over $p$-adic fields have also been studied; see \cite{Prasad2001} for $\GL_4$, and   \cite{Sanjeev2024} for $\Sp_4$.
Over finite fields, the focus has primarily been on the case where $\pi$ is cuspidal.
For $l>1$, a broader conjecture of Prasad proposes a description of $\pi_{N, \psi}$ for strongly cuspidal representation $\pi$ of $\GL_{2n}(\cO_l)$ \cite[Conjecture 1.4]{Ankita-Shiv}.
This has been verified for $n=2, l=2$ by the authors (see {\it loc. cit.} Theorem 7.11), but it remains open in general.

Prasad has also analyzed some principal series representations $\pi$ of $\GL_4(\F_q)$, providing the following theorem \cite[Theorem 4]{Prasad2000} and we will be considering analogous representations of $\GL_4(\cO_2)$.
\begin{theorem} [Prasad] \label{DP theorem} 
The degenerate Whittaker space of $\Ps(\pi_1,\pi_2)$, where $\pi_1$ and $\pi_2$ are irreducible representations of $\GL_2(\F_q)$ (neither of which is 1-dimensional) with central characters $\omega_{\pi_1}$ and $\omega_{\pi_2}$, is
\begin{center}
    $(\pi_1\otimes\pi_2)\oplus \Ps(\omega_{\pi_1},\omega_{\pi_2})$
\end{center}
where $\Ps(\omega_{\pi_1},\omega_{\pi_2})$ is the principal series representation of $\GL_2 (\F_q)$ induced from the character $\omega_{\pi_1} \otimes \omega_{\pi_2}$ of $\F_q^\times\times\F_q^\times$.
\end{theorem}

Assume $l>1$ and fix a non-trivial additive character \( \psi_0 : \mathfrak{o}_l \to \mathbb{C}^\times \) such that \( \psi_0|_{\varpi^{l-1} \mathfrak{o}_l} \neq 1 \). 
Define a character $\psi : N \rightarrow \C^{\times}$ given by $\psi(X)= \psi_{0}(\tr(X))$.
Then the $(N, \psi)$-Whittaker space $\pi_{N, \psi}$ of a representation $\pi$ of $\GL_{2n}(\cO_l)$ is a representation of $M_{\psi} \cong \GL_n(\cO_l)$, which is embedded diagonally inside $M$.
This paper is concerned with the case $l=2$ and $n=2$, i.e. describing the degenerate Whittaker space for representations of $\GL_4(\cO_2)$ as a representation of $\GL_2(\cO_2)$.
Since the strongly cuspidal representations were addressed in our earlier work \cite{Ankita-Shiv}, we focus here exclusively on some induced representations of $\GL_4(\cO_2)$. 
Consider the following subgroups $P$ and $Q$ of $\GL_4(\cO_2)$; 
\[ \begin{array}{ll}
    P= &  \left\lbrace\begin{pmatrix}
    g_1 & X\\
    0 & g_2
\end{pmatrix}:g_1,g_2\in \GL_2(\cO_2),~X\in M_2(\cO_2) \right\rbrace \cong M \ltimes N,\\
    Q= & \left\lbrace\begin{pmatrix}
    h_1 & Y\\
    0 & h_2
\end{pmatrix}:h_1\in \GL_3(\cO_2),~h_2\in \GL_1(\cO_2),~Y\in M_{3\times 1}(\cO_2)\right\rbrace \cong L \ltimes U
\end{array} \]
where $M \cong \GL_2(\cO_2) \times \GL_2(\cO_2), N \cong M_{2 \times 2}(\cO_2), L \cong \GL_3(\cO_2) \times \GL_1(\cO_2)$ and $U \cong M_{3 \times 1}(\cO_2)$.

Our main results are the following.
\begin{theorem} \label{induction from P-intro}
Let $\pi_1 = \Ind_{I(\phi_{B_1})}^{\GL_2(\cO_2)}  (\tilde{\phi}_{B_1})$ and $\pi_2 =  \Ind_{I(\phi_{B_2})}^{\GL_2(\cO_2)}  (\tilde{\phi}_{B_2})$ be strongly cuspidal representations of $\GL_2(\cO_2)$ with central characters $\omega_{\pi_1}$ and $\omega_{\pi_2}$ respectively (see Section \ref{construction of regualr representations--} for notation).
Let $\pi=\Ind_{P}^{\GL_4(\cO_2)}(\pi_1\otimes \pi_2)$, where $\pi_1 \otimes \pi_2$ is realized as a representation of $P$ via $P \twoheadrightarrow P/N \cong M$.
Assume that $\tr(B_1) \neq \tr(B_2)$ and $B  = \diag (\tr(B_1), \tr(B_2))$.
Then 
\[
\pi_{N,\psi} \cong 
(\pi_1\otimes \pi_2) ~~~~ \bigoplus ~~~~ \Ind_{\mathfrak{B}_2}^{\GL_2(\cO_2)} (\omega_{\pi_1} \otimes \omega_{\pi_2}) 
%~~~~ \bigoplus ~~~~ I(\omega_{\pi_1} \otimes \omega_{\pi_2}) 
~~~~ \bigoplus ~~~~ \Ind_{Z \cdot J^1_2}^{\GL_2(\cO_2)} (\omega_{\pi} \phi_B).
\]
\end{theorem}
This result is a variation of Theorem \ref{induction from P-intro} and its proof is detailed in Section \ref{Sec:4}.
We begin by restricting $\pi$ to the subgroup $P$ and analyzing $\pi|_{P}$ using Mackey theory.
We need a description of double cosets $P\backslash \GL_4(\cO_2) / P$ and get that there are six orbits in the Mackey theory (see Theorem \ref{6 cosets P-G-P}). 
%We explicitly compute the degenerate Whittaker spaces associated with each of these orbits and find that exactly three of them contribute non-trivially to $\pi_{N, \psi}$, appearing as direct summands as described above.
We explicitly compute the degenerate Whittaker spaces associated to all six orbits appearing in $\pi|_{P}$ and
find that exactly three of the six orbits contribute non-trivially to $\pi_{N, \psi}$, which appear in the direct summand of $\pi_{N, \psi}$ as in Theorem \ref{induction from P-intro}.

For $\GL_4(\F_q)$, Prasad studied $\pi_{N, \psi}$ for representations that are parabolic induction from the $(2,2)$ parabolic subgroup (see Theorem \ref{DP theorem}), but not for those induced from the $(3,1)$ parabolic, which apparently does not have a nice description.
In contrast, we also describe $\pi_{N, \psi}$ for representations induced from $Q$ as well.
More precisely, we prove the following.

\begin{theorem} \label{induction from Q-intro}
Let $\rho$ be a strongly cuspidal representation of $\GL_3(\cO_2)$ and $\chi$ a character of $\GL_1(\cO_2) \cong \cO_2^\times$.
Let $\pi = \Ind_Q^{\GL_4(\cO_2)}(\rho \otimes \chi)
$, where $\rho \otimes \chi$ is realized as a representation of $Q$ via $Q \twoheadrightarrow Q/U \cong L$.
Let $m_0 \in \F_q$ be such that 
$\psi_0(2m_0 x) = \omega_{\pi}(1+ \varpi \tilde{x})$
for all $x \in \F_q$, where $\tilde{x} \in \cO_2$ is any lift of $x$. 
Then 
$$
\pi_{N,\psi} \cong \bigoplus_{B} \Ind_{Z \cdot J^1_2}^{\GL_2(\cO_2)}(\omega_{\pi} \phi_B)
$$
where $B$ varies over the set of equivalence classes of all regular elements of $M_2(\F_q)$ with trace $2m_0$.
\end{theorem}
We prove the above theorem in Section \ref{Sec : 5}. Again, we restrict $\pi$ to $P$ and apply Mackey theory to analyze  $\pi|_{P}$.
The double coset decomposition $P\backslash \GL_4(\cO_2) / Q$ gives three orbits.
Explicit computations of the degenerate Whittaker spaces for each orbit show that only one of them contributes non-trivially.
Note that the resulting description of $\pi_{N, \psi}$ in Theorem \ref{induction from Q-intro} is particularly simple.
We further deduce the following corollaries.
\begin{corollary}
Let $\pi$ be as in Theorem \ref{induction from Q-intro}. Then
\begin{enumerate}[label = ($\alph*$)]
\item The degenerate Whittaker space $\pi_{N, \psi}$ consists of all the regular representations of $\GL_2(\cO_2)$ with central character as $\omega_{\pi}$.
\item $\pi_{N, \psi}$ is a multiplicity-free representation. 
\end{enumerate}  
\end{corollary}
\begin{corollary}
Let $\pi_1 = \Ind_Q^{\GL_4(\cO_2)}(\rho_1 \otimes \chi_1)$ and $\pi_2 = \Ind_Q^{\GL_4(\cO_2)}(\rho_2 \otimes \chi_2)$, where $\rho_1, \rho_2$ are strongly cuspidal representations of $\GL_3(\cO_2)$ and $\chi_1, \chi_2$ are characters of $\cO_2^{\times}$ (as in Theorem \ref{induction from Q-intro}). If the central characters of $\pi_1$ and $\pi_2$ are the same, then $(\pi_1)_{N, \psi} \cong (\pi_2)_{N, \psi}$ as representations of $\GL_2(\cO_2)$. 
\end{corollary}
We end the introduction by mentioning that in Section \ref{Sec:2} we recall required preliminaries, 
in Section \ref{Sec:4} we prove Theorem \ref{induction from P-intro} and in Section \ref{Sec : 5} we prove Theorem \ref{induction from Q-intro}.

\section{Preliminaries }\label{Sec:2}
Let $F$ be a finite unramified extension of $\Q_p$ or $\F_{p}((t))$. 
Let $\mathfrak{o}$ be its ring of integers with a uniformizer $\varpi$. 
For any positive integer $l$, define $\mathfrak{o}_l:=\mathfrak{o}/(\varpi^l)$. 
Note that $\mathfrak{o}/(\varpi) \cong \mathbb{F}_q$, where $\F_{q}$ is a finite field of order $q=p^f$ and characteristic $p>0$. 
For any integer \( m \) with \( 1 \leq m < l \), the natural projection 
$\mathfrak{o}_l \rightarrow \mathfrak{o}_m$,
induces a projection
$\GL_n(\mathfrak{o}_l) \rightarrow \GL_n(\mathfrak{o}_m)$.
For $g \in \GL_n(\mathfrak{o}_l)$, we denote its image in $\GL_n(\mathfrak{o}_m)$ by $\bar{g}$. 
For any \( h \in \GL_n(\mathfrak{o}_m) \), we denote by \( \tilde{h} \in \GL_n(\mathfrak{o}_l) \) a preimage of \( h \).
Define the $m$-the principal congruence subgroup of $\GL_n(\mathfrak{o}_l)$ as
$K_l^m := \mathrm{Ker}\left( \GL_n(\mathfrak{o}_l) \to \GL_n(\mathfrak{o}_m) \right)$.
Furthermore, if $m \geq l/2$, then $K_l^m$ is abelian.
This yields the following natural filtration:
$$
\{ I_n \} \subseteq K_l^{l-1} \subseteq \cdots \subseteq K_l^m \subseteq \cdots \subseteq K_l^1 \subseteq \GL_n(\mathfrak{o}_l).
$$
For each \( m \) such that \( 1 \leq m \leq l - 1 \), the quotient satisfies $K_l^m / K_l^{m+1} \cong (M_n(\mathbb{F}_q), +)$.
Let \( \widehat{M_n(\mathfrak{o}_l)} \) denote the Pontryagin dual of \( (M_n(\mathfrak{o}_l), +) \). 
Fix a non-trivial additive character \( \psi_0 : \mathfrak{o}_l \to \mathbb{C}^\times \) such that \( \psi_0|_{\varpi^{l-1} \mathfrak{o}_l} \neq 1 \).
For each \( A \in M_n(\mathfrak{o}_l) \), define  \( \psi_A \in \widehat{M_n(\mathfrak{o}_l)} \) by
\[
\psi_A(B) := \psi_0(\mathrm{tr}(AB)).
\]
The map
$A \mapsto \psi_A$ is an isomorphism
of \( M_n(\mathfrak{o}_l) \rightarrow \widehat{M_n(\mathfrak{o}_l)} \), which depends on the choice of \( \psi_0 \).

\subsection{Regular representations of $\GL_n(\cO_l)$}\label{construction of regualr representations--}
\begin{definition}
\begin{enumerate}
\item An element $x \in M_{n}(\F_{q})$ is called regular if its characteristic polynomial coincides with its minimal polynomial. 
\item For $x \in M_{n}(\F_q)$, the character $\phi_x : M_n(\mathbb{F}_q) \rightarrow \mathbb{C}^{\times}$ is called regular if $x$ is a regular element.
\item
A representation $\rho$ of $\GL_n(\cO_l)$ is said to be regular if its restriction to $K^{l-1}_l\cong M_n(\F_q)$ contains regular characters.
\end{enumerate}
\end{definition}
In the representation theory of \( \GL_n(\mathfrak{o}_l) \),  {\it{regular representations}} play a central role. 
In the case where \( l \) is even, a construction of regular representations was given by Hill~\cite{Hill1995}. This foundational work was later extended by Krakovski, Onn and Singla~\cite{Singla2018}, who developed a construction that works for all $l$.
In what follows, we recall the construction of regular representations of \( \GL_n(\mathfrak{o}_l) \) under the assumption that \( l \) is even.

\begin{theorem} \label{Construction of regualr reprresentation}
Let $l > 1$ be an even integer and set
$m = l/2$. 
Let $\rho : K^m_l \rightarrow \C^{\times}$ be a character such that $\rho|_{K^{l-1}_l}$ is a regular character. 
Let the inertia group of $\rho$ be
$$
I(\rho) : = \left\lbrace g\in \GL_n(\mathfrak{o}_l):\rho (g^{-1}xg)=\rho(x)~\forall ~x\in K^m_l\right\rbrace.
$$
Then the following hold.
\\
(1) The character $\rho$ extends to its inertia group $I(\rho)$.
If $\tilde{\rho}$ is an extension of $\rho$ to $I(\rho)$ then $\Ind_{I(\rho)}^{\GL_n(\mathfrak{o}_l)}\tilde{\rho}$ is an irreducible regular representation of $\GL_n(\mathfrak{o}_l)$.\\
(2) For a given regular representation $\pi$ of $\GL_n(\mathfrak{o}_l)$ such that $\rho$ appears in $\pi|_{K^m_l}$, then there exists $\tilde{\rho}$ a character in $I(\rho)$ such that $\pi=\Ind_{I(\rho)}^{\GL_n(\mathfrak{o}_l)}\tilde{\rho}$.
\end{theorem}

\begin{remark}
For $l=2m$, $K^{m}_{l} \cong M_{n}(\cO_{m})$ and any character of $M_{n}(\cO_{m})$ is $\phi_{B}$ for some $B \in M_{n}(\cO_{m})$.
For a regular representation $\pi$ of $\GL_{n}(\cO_l)$, if $\phi_{B}$ appears in  $\pi|_{K^{m}_{l}}$ then $\phi_{B'}$ appears in $\pi|_{K^{m}_{l}}$ if and only if $B' = gBg^{-1}$ for some
$g \in \GL_n(\cO_m)$. 
\end{remark}

\subsection{Regular representations of $\GL_2(\cO_2)$} \label{reg reps of GL_2}
We now describe all irreducible regular representations of $\GL_2(\cO_2)$.  
By Theorem~\ref{Construction of regualr reprresentation}, every irreducible regular representation of $\GL_2(\cO_2)$ is of the form
$\Ind_{I(\phi_B)}^{\GL_2(\cO_2)} \tilde{\phi}_B$,
where $B$ is a regular element in $M_2(\F_q)$ and $\tilde{\phi}_{B}$ is an extension of the character $\phi_{B}$ of $J^1_2:= I_2 + \varpi M_2(\cO_2)$ to the inertia group $I(\phi_{B})$.
Since any non-scalar element in $M_2(\F_q)$ is a regular element, the following three sets form a complete set of representatives for the conjugacy classes of regular elements in \( M_2(\mathbb{F}_q) \):
{\small
\begin{align*}
\mathcal{X}_1 &:= \left\{
\begin{pmatrix}
m & n\alpha\\
n & m
\end{pmatrix} : m \in \mathbb{F}_q,\ n \in \mathbb{F}_q^\times \right\}, 
\mathcal{X}_2 := \left\{
\begin{pmatrix}
m & 1\\
0 & m
\end{pmatrix} : m \in \mathbb{F}_q \right\},
\end{align*}
\[ 
\mathcal{X}_3 := \left\{
\begin{pmatrix}
m & 0\\
0 & n
\end{pmatrix} : m,n \in \mathbb{F}_q,\ m \neq  n \right\}.
\]
}
Therefore, any irreducible regular representation of $\GL_2(\cO_2)$ is of the form
$\Ind_{I(\phi_B)}^{\GL_2(\cO_2)} \tilde{\phi}_B$,
where $B$ ranges over  $\mathcal{X}_1 \cup \mathcal{X}_2 \cup \mathcal{X}_3$.

\subsection{Mackey theory}
Let $H$ and $K$ be subgroups of a finite group $G$. 
Let $\rho$ be a representation of $H$. 
Then Mackey theory describes the restriction of the induced representation $\Ind_{H}^{G} \rho$ of $G$ to the subgroup $K$, denoted $\Res_{K}^{G} \Ind_{H}^{G} \rho$, see \cite{Serre} for more details.
This description is as follows.
First, choose a set $S$ of representatives for the double cosets $K\backslash G/H$, so that $G=\underset{s\in S}{\bigsqcup }KsH$. 
For $s\in S$, define the subgroup $H_s :=sHs^{-1}\cap K$ of $K$. 
Define
$$
\rho^s(x)=\rho(s^{-1}xs) \quad \text{for all } x\in H_s.
$$
Then $\rho^s$ is a representation of $H_s$ and the restriction of $\Ind_{H}^{G}\rho$ to $K$ decomposes as 
\begin{equation}\label{2.1 equation for mackey theory}
\Res^{G}_{K}\Ind_{H}^{G}\rho \cong \underset{s\in K\backslash G/H}{\bigoplus}\Ind_{H_s}^{K}\rho^{s}.
\end{equation}

\subsection{Multiplicity}
For finite dimensional representations $\rho$ and $\rho'$ of a finite group $H$, we define 
\[
m( \rho, \rho') := \dim \left(\Hom_{H}(\rho, \rho')\right).
\]
If $\rho'$ is irreducible then $m(\rho, \rho')$ is said to be multiplicity of $\rho'$ in $\rho$, which counts the number of times $\rho'$ appears in $\rho$.
Let $\pi$ and $\sigma$ be irreducible representations of $\GL_{2n}(\cO_l)$ and $\GL_n(\cO_l)$, respectively. Then 
\begin{equation}\label{equivalence in Hom for Ch-2}
\Hom_{\GL_n(\cO_l)}(\pi_{N,\psi}, \sigma) \cong \Hom_{\Delta \GL_n(\cO_l)\cdot N}(\pi, \sigma \otimes \psi).
\end{equation}
In particular, $m (\pi_{N,\psi}, \sigma) = m (\pi, \sigma \otimes \psi)$.

\subsection{A few double cosets}
The following double cosets will play an important role in what follows.
\begin{lemma}\label{double cosets P-1-n-2_GL_n-F_{q^n}}
Let $P_{1,n-1}$ be the parabolic subgroup of $\GL_n(\F_q)$ corresponding to the partition $(1,n-1)$ of $n$. Then
\begin{equation}\label{double cosets for GL_n}
\GL_n(\F_q)=P_{1,n-1}\cdot \F_{q^n}^\times.
\end{equation}
\end{lemma}

\begin{proof}
We know that the quotient $\GL_n(\mathbb{F}_q)/P_{1,n-1}$  is nothing but the Grassmannian $\Gr(n,1)$, i.e. the set of all $1$-dimensional subspaces in a $n$-dimensional vector space over $\F_q$. 
For a group $G$ acting on a set $X$ we write, \def\acts{\curvearrowright} $G \acts X$.
Naturally, \def\acts{\curvearrowright}
$\GL_n(\mathbb{F}_q)\acts$ Gr$(n,1)$ transitively and distinct orbits of the restricted action to the subgroup $\F_{q^n}^\times \hookrightarrow \GL_n(\mathbb{F}_q)$ corresponds to the distinct double cosets $\F_{q^n}^\times\backslash \GL_n(\F_q)/P_{1,n-1}$.
Now, consider $\F_{q^n}$ as a $n$-dimensional vector space over $\F_q$. 
Any element of Gr$(n,1)$ can be expressed as $\Span\{x\}$, the span of vector $x \in \F_{q^n}^\times$.
The action \def\acts{\curvearrowright} $\F_{q^n}^\times\acts$ Gr$(n,1)$ is transitive, since any nonzero vector can be mapped to any other nonzero vector by multiplication by a suitable element from $\F_{q^n}^{\times}$.
Thus there is a unique double coset in $\F_{q^n}^\times\backslash \GL_n(\F_q)/P_{1,n-1
}(\F_q)$, proving the lemma.
\end{proof}
Recall the natural quotient map \( \GL_n(\cO_2) \rightarrow \GL_n(\F_q) \).
For any subset \( X \subseteq \GL_n(\cO_2) \) and \( Y \subseteq \GL_n(\F_q) \), we write \( \bar{X} \) for the image of \( X \), and \( \tilde{Y} \) for the preimage of \( Y \) under the quotient map.
Let $\mathfrak{B}_n$ be the subgroup of upper triangular matrices in $\GL_n(\cO_2)$, and let $w_0:=\begin{pmatrix}
    0 & 1\\
    1 & 0
\end{pmatrix}$. 
Let $H_2(\F_q):=\left\{\begin{pmatrix}
    x & y\\
    0 & 1
\end{pmatrix}:x\in \F_q^\times, y\in \F_q\right\}$.
% \textcolor{blue}{From \cite[Section 3]{Amritanshu}, we obtain the following double coset decompositions.}
\begin{lemma}\label{GL_2 double cosets}
\begin{enumerate}
\item 
A set of representatives for $\bar{\mathfrak{B}}_2\backslash \GL_2(\F_q)/\bar{\mathfrak{B}}_2$ can be taken to be $
\left\{ I_2, w_0
\right\}$.
\item 
A set of representatives for $\mathfrak{B}_2\backslash \GL_2(\cO_2)/\mathfrak{B}_2$ can be taken to be 
$\left\{ I_2, w_0, \begin{pmatrix}
    1 & 0\\
    \varpi & 1
\end{pmatrix}
\right\}$.
\item A set of representatives for $\GL_2(\cO_2)/\mathfrak{B}_2$ can be taken to be $\left\{ \begin{pmatrix}
    1 & 0\\
    x & 1
\end{pmatrix},\begin{pmatrix}
    \varpi y & 1\\
    1 & 0
\end{pmatrix} : x,y\in \cO_2
\right\}$.
\item A set of  for $\GL_2(\F_q)/H_2(\F_q)$ can be taken to be $\left\{ \begin{pmatrix}
    1 & 0\\
    c & a
\end{pmatrix}, dw_0:
    a, d\in \F_q^\times, 
    c \in \F_q\right\}$.
\end{enumerate}
\end{lemma}
\begin{proof}
Part $(1)$ and $(2)$ follow from \cite[Section 3]{Amritanshu}.  
Part $(3)$ and $(4)$ follow from the observation $\GL_2(\F_q) = \F_{q^2}^{\times} \cdot \bar{\mathfrak{B}}_{2}$.
\end{proof}

The following proposition from \cite[Exercise 4.1.18]{Bump}
gives an easy description of certain double cosets which will be helpful in Section \ref{Sec:4}.
\begin{proposition}
\label{Double cosets proposition over finite field}
Let $\mathcal{P}_1$ and $\mathcal{P}_2$
be two standard parabolic subgroups of $\GL_n(\F_q)$.
Let $W$ be the subgroup
of $\GL_n(\F_q)$ consisting of permutation matrices.
Let $W_{\mathcal{P}_i} = W\cap \mathcal{P}_i$ for $i=1,2$.
The inclusion of $W$ in $G$ induces a bijection
between the double cosets $\mathcal{P}_2\backslash \GL_n(\F_q)/\mathcal{P}_1$ and $W_{\mathcal{P}_2}\backslash W/W_{\mathcal{P}_1}$.
\end{proposition}

\subsection{Certain embeddings} \label{all embeddings}
We now fix embeddings of the finite fields \(\F_{q^2}\), \(\F_{q^3}\) into the matrix algebras \(M_2(\F_q)\), \(M_3(\F_q)\) respectively. 
For the embeddings of $\F_{q^2}$ into $M_2(\F_q)$, we assume that the characteristic of $\F_q$ is not equal to 2.
% The embedding of $\F_{q^3}$ is used in Section \ref{Sec : 5}.

\begin{enumerate}[label = (\alph*)]
\item 
Let $\alpha \in \F_{q}^{\times} \smallsetminus \F_{q}^{\times 2}$. 
Then $\F_{q^2}=\F_{q}[\sqrt{\alpha}]$.
We fix the embedding $\F_{q^2} \hookrightarrow M_2(\F_q)$ as follows
\begin{equation}\label{standard F_q^2 embedding}
c+d\sqrt{\alpha} ~~ \mapsto ~~ \begin{pmatrix}
    c & d\alpha\\
    d & c
\end{pmatrix}.
\end{equation}

\item
For $a \in \F_q$, the map $\F_q \rightarrow \F_q$ given by $y \mapsto y^3-y-a$ is not one-to-one (0 and 1 have the same image!) hence not onto. 
Therefore, $\exists ~ a \in \F_q^{\times}$ such that the polynomial $y^3-y-a$ does not have a zero in $\F_q$ and hence irreducible. 
We fix $a \in \F_{q}^{\times}$ such that $y^3-y-a$ is irreducible.
%Let $a\in \F_q^\times $ be such that the polynomial $y^3-y-a$ is irreducible over $\F_q$.
Then $\F_{q^3} = \F_{q}[\zeta]$, where $\zeta$ is a root of 
$y^3-y-a$.
We fix the embedding of $\F_{q^3}  \hookrightarrow M_3(\F_q)$ as follows
\begin{equation}\label{Embedding of Fq^3}
    a_0+a_1\zeta +a_2\zeta^2 \mapsto\begin{pmatrix}
        a_0 & a_2a & a_1a\\
        a_1 & a_0+a_2 & a_1+a_2a\\
        a_2 & a_1 & a_2+a_0
    \end{pmatrix}.
\end{equation}
\end{enumerate}

\subsection{Irreducible components of $\mathcal{I}(\chi)$}
\label{Section 2.7}
Let $\chi_1, \chi_2 : \cO_2^{\times} \rightarrow \C^{\times}$ be two characters.
Let us define $\mathcal{I}(\chi) := \Ind_{\mathfrak{B}_2}^{\GL_2(\cO_2)} (\chi)$,
where $\chi = \chi_1 \otimes \chi_2 : \mathfrak{B}_2 \rightarrow \C^{\times}$ is given by
$\chi \left( \begin{matrix}
x & y \\ 
0 & z \end{matrix} \right) : = \chi_1(x) \chi_2(z)$.

\begin{lemma} \label{irreducibility of ps of gl2}
\begin{enumerate}[label = (\alph*)]
\item $\mathcal{I}(\chi)$ is irreducible if and only if $\chi_1|_{1+ \varpi \cO_2} \neq \chi_2|_{1+ \varpi \cO_2}$.
\item $\mathcal{I}(\chi)$ has exactly two irreducible components if and only if $\chi_1|_{1+ \varpi \cO_2} = \chi_2|_{1+ \varpi \cO_2}$ but $\chi_1 \neq \chi_2$.
\item $\mathcal{I}(\chi)$ splits into three irreducible components if and only if $\chi_1 = \chi_2$.
\end{enumerate}
\end{lemma}

\begin{proof}
By Mackey theory and Frobenius reciprocity, 
\begin{align*}
\Hom_{\GL_2(\cO_2)}( \mathcal{I}(\chi), \mathcal{I}(\chi) ) 
& \cong  \underset{\eta\in \mathfrak{B}_2\backslash \GL_2(\cO_2)/\mathfrak{B}_2}{\bigoplus} \Hom_{\eta \mathfrak{B}_2\eta^{-1}\cap \mathfrak{B}_2}\left(\chi^{\eta}, \chi\right).
\end{align*}
The lemma follows by using Lemma \ref{GL_2 double cosets} together with the following observations:
\\
For $\eta =I_2$, we have
$\Hom_{\eta \mathfrak{B}_2\eta^{-1}\cap \mathfrak{B}_2}\left(\chi^{\eta}, \chi\right) \cong \C$.   
\\
For $\eta = w_0$, we have $\Hom_{\eta \mathfrak{B}_2\eta^{-1}\cap \mathfrak{B}_2}\left(\chi^{\eta}, \chi\right) \cong \C$ if and only if $\chi_1 = \chi_2$ and $0$ otherwise.
\\
For $\eta = \begin{pmatrix}
    1 & 0\\
    \varpi & 1
\end{pmatrix}$,
we have  $\Hom_{\eta \mathfrak{B}_2\eta^{-1}\cap \mathfrak{B}_2}\left(\chi^{\eta}, \chi\right) \cong \C$ if and only if $\chi_1|_{1+\varpi \cO_2} = \chi_2|_{1+\varpi \cO_2}$ and $0$ otherwise.
\end{proof}

\begin{lemma}\label{Lemma 2.1.7}
If $\chi_1|_{1+ \varpi \cO_2} = \chi_2|_{1+ \varpi \cO_2}$, then  $\Ind_{\mathfrak{B}_2}^{\tilde{\mathfrak{B}}_2} \chi$ decomposes into two irreducible components of dimensions 1 and $q-1$.
\end{lemma}

\begin{proof}
By Mackey theory, it is simple to check that $\Ind_{\mathfrak{B}_2}^{\tilde{\mathfrak{B}}_2} \chi$ has exactly two irreducible components under the given condition. 
Indeed, $\chi$ extends to $\tilde{\mathfrak{B}}_2$ via $\begin{pmatrix}
    x & y\\
    \varpi z & w
\end{pmatrix} \mapsto \chi_1(x) \chi_2(w)$ giving the one dimensional representation and the lemma follows.
\end{proof}

Note that  $\dim \mathcal{I} (\chi) = q(q+1)$.    
Now we describe the dimensions of the irreducible components of $\mathcal{I}(\chi)$. 

\begin{proposition}\label{Proposition 2.1.6}
\begin{enumerate}[label = (\alph*)]
\item If $\chi_1|_{1+ \varpi \cO_2} = \chi_2|_{1+ \varpi \cO_2}$ but $\chi_1 \neq \chi_2$ then the dimensions of the two irreducible components of 
$\mathcal{I} (\chi)$ are $q+1$ and $q^2 -1$.
\item If $\chi_1 = \chi_2$ then the  dimensions of the three irreducible components of $\mathcal{I}(\chi)$ are $1, q$ and $q^2 -1$.
\end{enumerate}    
\end{proposition}

\begin{proof}
By previous lemma, let us write $\Ind_{\mathfrak{B}_2}^{\tilde{\mathfrak{B}_2}}(\chi) \cong \tilde{\chi} \oplus \tau$, where $\dim(\tilde{\chi})=1$ and $\dim (\tau)=q-1$.
Then, by the transitivity of induction, we obtain
\[
\mathcal{I}(\chi)  \cong \Ind_{\tilde{\mathfrak{B}}_2}^{\GL_2(\cO_2)} \tilde{\chi} \oplus \Ind_{\tilde{\mathfrak{B}}_2}^{\GL_2(\cO_2)} \tau. 
\]
\begin{enumerate}[label = ($\alph*$)]
\item 
If $\chi_1 \neq \chi_2$, then Lemma \ref{irreducibility of ps of gl2} implies that $\mathcal{I}(\chi)$ has exactly two irreducible components. 
Therefore, $\sigma_1 := \Ind_{\tilde{\mathfrak{B}}_2}^{\GL_2(\cO_2)} \tilde{\chi}$ and $\sigma_2 := \Ind_{\tilde{\mathfrak{B}}_2}^{\GL_2(\cO_2)} \tau$ are irreducible. 
Since $[\GL_2(\cO_2) : \tilde{\mathfrak{B}}_2] = q+1$, we have $\dim(\sigma_1) =q+1$ and $\dim (\sigma_2) =q^2 -1$. 
\item If $\chi_1 = \chi_2$,  Lemma \ref{irreducibility of ps of gl2} shows that $\mathcal{I}( \chi)$ splits into three irreducible components. 
Moreover, it can be verified that $\Ind_{\tilde{\mathfrak{B}}_2}^{\GL_2(\cO_2)} \tilde{\chi}$ decomposes into two irreducible components, say $\sigma'_1$ and $\sigma'_2$, where $\dim(\sigma'_1)=1$ and $\dim( \sigma'_2) =q$.
Indeed, $\sigma'_1 = \chi_1 \circ \det$.
We conclude that $\mathcal{I}(\chi) = \sigma'_1 \oplus \sigma'_2 \oplus \sigma'_3$, where $\sigma'_3 := \Ind_{\tilde{\mathfrak{B}}_2}^{\GL_2(\cO_2)} \tau$. \qedhere 
\end{enumerate}
\end{proof}

% {\color{blue}The following notation will be used consistently throughout the paper.}
% \begin{notation} \label{some important notations}
% %\begin{enumerate}
% %\item 
% Let $u^+(x) := \begin{pmatrix}
%     1 & x\\
%     0 & 1
% \end{pmatrix}$ and $u^-(y) := \begin{pmatrix}
%     1 & 0\\
%     y & 1
% \end{pmatrix}$.
% %\end{enumerate}
% \end{notation}

\section{$\pi_{N,\psi}$ for $\pi=\Ind_{P}^{\GL_4(\cO_2)}(\pi_1\otimes \pi_2)$}\label{Sec:4}

For $i=1,2$, we fix strongly cuspidal irreducible representations $\pi_i = \Ind_{I(\phi_{B_{i}})}^{\GL_2(\cO_2)} \tilde{\phi}_{B_i}$ of $\GL_2(\cO_2)$, where $B_i=\begin{pmatrix}
    m_i & n_i\alpha\\
    n_i & m_i
\end{pmatrix}$ with $m_i \in \F_{q}$, $n_{i} \in \F_{q}^{\times}$ and $\alpha \in \F_{q}^{\times}$ is a fixed non-square element. 
In this section, we study the degenerate Whittaker space $\pi_{N, \psi}$ for $\pi=\Ind_{P}^{\GL_4(\cO_2)}(\pi_1\otimes \pi_2)$, where $\pi_1 \otimes \pi_2$ is a representation of $P$ via $P \rightarrow P/N \cong \GL_2(\cO_2) \times \GL_2(\cO_2)$.
Restricting $\pi$ to $P$, Mackey theory gives the following decomposition
\begin{center}
$ \pi|_{P} \cong \underset{ \delta } {\bigoplus} \pi^{\delta} $
\end{center} 
where $\delta$ varies over a set of representatives for the double cosets $P \backslash \GL_4(\cO_2)/P$ and $\pi^{\delta} := \Ind_{\delta P \delta^{-1} \cap P}^{P} (\pi_1\otimes\pi_2)^{\delta}$. 
For a subgroup $H \subseteq G$, let $\Delta H \subset G \times G$ denote the diagonal embedding of $H$ in $G \times G$.
Since $\Delta \GL_2(\cO_2) \cdot N \subseteq P$, we get
\begin{equation}\label{3.1 equation-}
\pi_{N, \psi} 
\cong (\pi|_{P})_{N, \psi} 
\cong \underset{\delta\in P\backslash \GL_4(\cO_2)/P}{\bigoplus}\pi^{\delta}_{N, \psi}
\cong \underset{\delta\in P\backslash \GL_4(\cO_2)/P}{\bigoplus}\Hom_{N}(\pi^{\delta},\psi).  
\end{equation} 
We describe $\pi^{\delta}_{N,\psi}$ as a representation of $\GL_2(\cO_2)$ for every $\delta \in P\backslash \GL_4(\cO_2)/P$.

\subsection{A description of $P\backslash \GL_4(\cO_2)/P$ }
We write $\diag(a_1, a_2,...,a_k)$ for the diagonal matrix with diagonal entries $a_1, a_2, ...,a_k$.
Now, we describe the double cosets $P\backslash \GL_4(\cO_2)/P$.

\begin{theorem}\label{6 cosets P-G-P}
The number of distinct double cosets for $P\backslash \GL_4(\cO_2)/P$ is $6$.
A set of distinct representatives of the double coset {\small{$P\backslash \GL_4(\cO_2)/P$ is given by
{\footnotesize{\begin{center}
$\left\{
\delta
_1=I_4,\delta
_2=\begin{pmatrix}
  I_2 & 0 \\
  \diag(\varpi , 0) & I_2
\end{pmatrix}
,\delta_3=\begin{pmatrix}
  I_2 & 0 \\
\varpi I_2 & I_2
\end{pmatrix},
\delta_4=\begin{pmatrix}
  0 & I_2\\
  I_2 & 0
\end{pmatrix},
\delta_5= \begin{pmatrix}
    1 & 0 & 0\\
    0 & w_0 & 0 \\
    0 & 0 & 1
\end{pmatrix}
,
\delta_6 = \begin{pmatrix}
    1 & 0 & 0\\
    0 & w_0 & 0 \\
    \varpi & 0 & 1
\end{pmatrix}
 \right\}$.
\end{center}}}}}
\end{theorem}

\begin{proof}
Note that the natural map $\GL_4(\cO_2) \twoheadrightarrow \GL_4(\mathbb{F}_q)$ induces a surjective map
\begin{equation}\label{surjective map}
P \backslash \GL_4(\cO_2) / P \twoheadrightarrow \bar{P} \backslash \GL_4(\mathbb{F}_q) / \bar{P}.
\end{equation}
Using Proposition 
\ref{Double cosets proposition over finite field}, there exists a bijection between $\bar{P}\backslash \GL_4(\F_q) / \bar{P}$ and $W_{\bar{P}} \backslash W / W_{\bar{P}}$, where $W_{\bar{P}} \cong W \cap \bar{P}$ and $W$ is the Weyl group of $\GL_4(\F_q)$ which can be identified with the set of permutation matrices in $\GL_4(\mathbb{F}_q)$.
Observe that $W_{\bar{P}} = 
\{
I_4,(12),(34),(12)(34)
\}$, where  $(12)$ (respectively, $(34)$) denote the permutation matrices obtained by interchanging the first and the second rows (respectively, the third and the forth rows) of $I_4$.
It can be seen that
\begin{center}
$\left\lbrace I_4, (13), (13)(24)\right\rbrace$
\end{center}
is a set of distinct representatives of the double cosets
$W_{\bar{P}}\backslash W/W_{\bar{P}}$.
Therefore, using the map in \eqref{surjective map}, an exhaustive set of representatives for $P\backslash \GL_4(\cO_2)/P$ is
\begin{equation}\label{exhaustive set}
\left\lbrace I_4+\varpi A,  (13)+\varpi B,(13)(24)+\varpi C:A,B,C\in M_4(\mathfrak{o}_2)\right\rbrace.
\end{equation}
It can be verified that the six representatives mentioned in the theorem form a complete set of distinct representatives for  $P\backslash \GL_4(\cO_2)/P$.
\qedhere
\end{proof}

\subsection{$\pi^{\delta_i}_{N, \psi}= 0$ for $i=1,2,3$}
\begin{proposition} \label{deltas with 0 WM}
If $\delta\in \left\lbrace\delta_1,\delta_2,\delta_3\right\rbrace$, then $\pi^{\delta}_{N,\psi}=0$.
\end{proposition}

\begin{proof}
Recall that $\pi^{\delta}=\Ind_{\delta P\delta^{-1}\cap P}^{P}(\pi_1\otimes \pi_2)^{\delta}$.
For $\delta = \delta_1 =I_4$, $\pi^{\delta} = \pi_1 \otimes \pi_2$ on which $N$ acts trivially. 
On the other hand, the character $\psi$ of $N$ is non-trivial. 
Hence using \eqref{3.1 equation-} $\pi^{I_4}_{N, \psi} =0$.

Since \( \pi^{\delta} = \Ind_{\delta P \delta^{-1} \cap P}^P (\pi_1 \otimes \pi_2)^{\delta} \), using Mackey theory we get
\begin{align} \label{equation for pi-delta-N,psi for trivial doubel coset}
\Hom_{N}(\pi^{\delta}, \psi) 
& \cong \underset{\gamma \in N \backslash P / \delta P \delta^{-1} \cap P}{\bigoplus}\Hom_{N} \left( \Ind_{\gamma (\delta P \delta^{-1} \cap P) \gamma^{-1} \cap N}^N (\pi_1 \otimes \pi_2)^{\delta \gamma}, \psi \right) \nonumber \\
& \cong \underset{\gamma \in N \backslash P / \delta P \delta^{-1} \cap P}{\bigoplus}\Hom_{N} \left(  (\pi_1 \otimes \pi_2)^{\delta}, \psi^{\gamma^{-1}} \right)
\end{align}
where the last equality follows from $\gamma (\delta P \delta^{-1} \cap P) \gamma^{-1} \cap N = N$ for any $\gamma \in P$.
Let 
$$
N_1 := \left\lbrace \begin{pmatrix}
    I_2 & \varpi X\\
    0 & I_2
\end{pmatrix}: X \in M_2(\cO_2) \right \rbrace \subset N.
$$
Now we consider $\delta \in \{ \delta_2, \delta_3 \}$.
Then $(\pi_1 \otimes \pi_2)^{\delta}$ is trivial on $N_1$.
Since $\psi$ is non-trivial on $N_1$, it follows that $\psi^{\gamma^{-1}}$ is non-trivial on $N_1$ for any $\gamma \in P$.
Therefore, $\Hom_{N_1} \left(  (\pi_1 \otimes \pi_2)^{\delta}, \psi^{\gamma^{-1}} \right) =0$ for any $\gamma \in P$.
Since $N_1 \subset N$, by \eqref{equation for pi-delta-N,psi for trivial doubel coset}, we get $\pi^{\delta}_{N, \psi} = \Hom_{N}(\pi^{\delta}, \psi) = 0$.
\end{proof}

\subsection{A description of $\pi^{\delta_4}_{N,\psi}$}

\begin{lemma}\label{delta_4 coset}
The representation $\pi^{\delta_4}_{N,\psi} \cong \pi_1\otimes\pi_2$.
\end{lemma}
\begin{proof}
Recall that $\pi^{\delta_4} = \Ind_{\delta_4 P \delta_4^{-1} \cap P}^{P}(\pi_1\otimes\pi_2)^{\delta_4}$ and $\delta_4 P \delta_4^{-1}\cap P \cong \GL_2(\cO_2)\times\GL_2(\cO_2)$. 
Therefore, 
\begin{center}
$\Res^{P}_{\Delta \GL_2(\cO_2) \cdot N} \pi^{\delta_4} \cong \Ind_{\Delta \GL_2(\cO_2)}^{\Delta \GL_2(\cO_2) \cdot N} (\pi_1\otimes\pi_2 )
\cong (\pi_1 \otimes \pi_2) \otimes \Ind_{\{I_4\}}^{N} \mathbb{C}$.
\end{center}
Note that $\Ind_{\{I_4\}}^{N} \mathbb{C} \cong \mathbb{C}[N]$ contains all the characters of $N$ exactly once.
Since $\Delta \GL_2(\cO_2)$ stabilizes the character $\psi$ of $N$, we get
$\pi^{\delta_4}_{N, \psi} \cong \pi _1 \otimes \pi_2$.
\end{proof}

\subsection{A description of $\pi^{\delta_5}_{N,\psi}$}

Let us write $u^+(x) := \begin{pmatrix}
    1 & x\\
    0 & 1
\end{pmatrix}$ and $u^-(y) := \begin{pmatrix}
    1 & 0\\
    y & 1
\end{pmatrix}$.

\begin{lemma}\label{tensor product on hom}
Let $\pi_1$ and $\pi_2$ be regular representations of $\GL_2(\cO_2)$ and let $N' : = \left\{ u^+(x)
: x \in \cO_2 \right\}$.  
Define  a character $\psi$ of $N' \times N'$ as $\psi := \psi_{0} \otimes \psi_{0}$. 
Then $\Hom_{N'\times N'}(\pi_1\otimes \pi_2, \psi)\cong \C$.
\end{lemma}
\begin{proof}
The proof follows from
$$
\Hom_{N' \times N'}(\pi_1 \otimes \pi_2, \psi_1 \otimes \psi_2)\cong \Hom_{N'}(\pi_1,\psi_1)\otimes \Hom_{N'}(\pi_2, \psi_2),
$$
and using the uniqueness of Whittaker models for $\pi_1$ and $\pi_2$, see \cite{SP2022}.
\end{proof}
Since there is a natural surjection $\cO_2\rightarrow\cO_{1}=\F_q$ written as $x\mapsto\bar{x}$, we have, $ \F_q=\cO_{1}\cong \varpi \mathfrak{o}_2\subseteq \cO_2$.
Given a character $\psi_0:\cO_2\rightarrow\C^\times$, its restriction to the subgroup $\varpi \cO_2$ can be seen as character of $\cO_{1}$ again denoted by $\psi_{0}$ satisfying the following property
$\psi_0(\varpi a)=\psi_0(\bar{a})$ for all $a\in \cO_{2}$. 
This isomorphism depends on the choice of uniformizer $\varpi$.

We denote a block diagonal matrix with diagonal blocks $A_1, A_2,..., A_k$ by $\mathrm{Diag}(A_1,A_2,..., A
_k)$.
\begin{lemma} \label{dim of pi delta_5}
Let $\delta=\delta_5$. 
Then  $\dim(\pi^{\delta}_{N,\psi})=q(q+1)$.
\end{lemma}
\begin{proof}
We know
$\pi^{\delta_5}_{N,\psi}= \Hom_{N}\left(\Ind_{\delta_5 P\delta_5^{-1}\cap P}^{P}(\pi_1\otimes \pi_2)^{\delta_5}, \psi\right)$.
Using Mackey theory and the fact that $N$ is a normal subgroup of $P$, we get 
\begin{align}\label{4.7}
\Hom_{N}\left(\Res_{N}^{P}\Ind_{\delta_5 P\delta_5^{-1}\cap P}^{P}(\pi_1\otimes \pi_2)^{\delta_5}, \psi\right)
 & \cong \underset{\gamma\in \Gamma}{\bigoplus}~ W_{\gamma}, 
 % \Hom_{\delta_5 P\delta_5^{-1}\cap N}\left((\pi_1\otimes \pi_2)^{\delta_5},\psi^{\gamma^{-1}}\right)
\end{align}
where $W_{\gamma}:=\Hom_{\delta_5 P\delta_5^{-1}\cap N}\left((\pi_1\otimes \pi_2)^{\delta_5},\psi^{\gamma^{-1}}\right)$ and $\Gamma$ is a set of representatives of $N\backslash P/\delta_5 P\delta_5^{-1}\cap P$, which can be taken to be 
$$
\left\{ \Diag(g,h)
: g,h\in \left\{{u^-(a)},\begin{pmatrix}
    \varpi b & 1\\
    1 & 0
\end{pmatrix} : a,b\in \cO_2\right\}\right\}
$$
We write $P_{\delta_5}=\delta_5 P\delta_5^{-1}\cap P$, we have
$$ 
P_{\delta_5}\cap N=
\left\lbrace
Y : =\begin{pmatrix}
    I_2 & \begin{pmatrix}
        p_{13} & p_{14}\\
        0 & p_{24}
    \end{pmatrix} \\
    0  & I_2
\end{pmatrix} ~{\Big{|}}~  p_{13}, p_{14}, p_{24}\in \cO_2\right\rbrace.
$$
Let $N_0$ be the subgroup of $P_{\delta_5}$ as follows 
\begin{equation}\label{The subgoup N_0}
N_0:=\left\lbrace n_0:=\begin{pmatrix}
    I_2 & 
    u^+(p_{14}) \\
    0 & I_2
\end{pmatrix}~{\big{|}}~ p_{14}\in \cO_2
\right\rbrace
\end{equation}
For $Y \in  P_{\delta_5}\cap N$, we have 
\begin{equation}\label{pi_1otimes pi_2 twisted by delta_5 action}
(\pi_1\otimes \pi_2)^{\delta_5}(Y)=
\pi_1(u^+(p_{13}))
\otimes 
\pi_2(u^+(p_{24})).
\end{equation}
We now divide the proof into cases to compute $\dim(W_{\gamma})$ for all $\gamma\in \Gamma$.
\\
Case 1 :
Let $\gamma=\Diag(g, h)$ with $g=u^-(a_1)$ and $h=u^-(a_2)$ where $a_1,a_2\in \cO_2$. 
Then
\begin{equation}\label{a_1=a_2, character value}
\psi^{\gamma^{-1}}(Y)=\psi_0(p_{24}+p_{13}+(a_1-a_2)p_{14}).
\end{equation}
Clearly, $(\pi_1\otimes \pi_2)^{\delta_5}$ is trivial on $N_0$, while $\psi^{\gamma^{-1}}(n_0)=\psi_0((a_1-a_2)p_{14})$ is non-trivial, if $a_1\neq a_2$.
Thus, if $a_1\neq a_2$,
$\Hom_{N_0}\left((\pi_1\otimes \pi_2)^{\delta_5},\psi^{\gamma^{-1}}\right)=0$ which implies $W_{\gamma}=0$.
Now, assume $a_1=a_2$.
Then using  \eqref{pi_1otimes pi_2 twisted by delta_5 action} and \eqref{a_1=a_2, character value}, together with Lemma \ref{tensor product on hom}, we get
$\dim\left(W_{\gamma}\right) =1$.
\\
Case 2 : Let $\gamma= \Diag(u^-(a), h)$ 
or 
$\Diag(h, u^-(a))$
with $h=\begin{pmatrix}
    \varpi b & 1\\
    1 & 0
\end{pmatrix}$. 
Then 
$$
\psi^{\gamma^{-1}}(n_0)=\psi_0((1-\varpi ab)p_{14}),
$$
which is a non-trivial on $N_0$. 
This gives
$\Hom_{N_0}\left((\pi_1\otimes \pi_2)^{\delta_5},\psi^{\gamma^{-1}}\right)=0$ and then $W_{\gamma}=0$. 
\\
Case 3 :
If $\gamma= \Diag(g,h)$
with $g=\begin{pmatrix}
    \varpi b & 1\\
    1 & 0
\end{pmatrix}$ and $h=\begin{pmatrix}
    \varpi c & 1\\
    1 & 0
\end{pmatrix}$, then we get
\begin{equation}\label{b=c redcution in computation of Hom}
\psi^{\gamma^{-1}}(Y)=\psi_0(p_{24}+p_{13}+\varpi (b-c)p_{14}).    
\end{equation}
If $\varpi b\neq \varpi c$, then $\psi^{\gamma^{-1}}$ is non-trivial on $N_0$, which gives
$\Hom_{N_0}\left((\pi_1\otimes \pi_2)^{\delta_5},\psi^{\gamma^{-1}}\right)=0$ and hence $W_{\gamma}=0$.
If $\varpi b =\varpi c$, then using \eqref{b=c redcution in computation of Hom}, \eqref{pi_1otimes pi_2 twisted by delta_5 action} and Lemma \ref{tensor product on hom}, we get
$\dim\left(W_{\gamma}\right)=1$.
\\
From Case 1, Case 2 and Case 3, we get \[
\dim(\pi^{\delta_5}_{N,\psi}) = \sum_{\gamma \in \Gamma} \dim (W_{\gamma}) = q^2 +q. \qedhere 
\] 
\end{proof}

\begin{proposition}
\label{dimension of all prinicpal series in pi-delta_5-N,psi}
Let  $\omega_{\pi_1}$ and $\omega_{\pi_2}$ be the central characters of the strongly cuspidal representations $\pi_1$ and $\pi_2$ of $\GL_2(\cO_2)$, respectively.
Let $\sigma=\Ind_{\mathfrak{B}_2}^{\GL_2(\cO_2)}(\omega_{\pi_1}\otimes \omega_{\pi_2})$.
Then the following holds.
\begin{enumerate}[label = (\roman*)]
\item 
If $\omega_{\pi_1}|_{1+\varpi \cO_2}\neq \omega_{\pi_2}|_{1+\varpi \cO_2}$, then 
$m \left(\pi^{\delta_5}_{N, \psi},\sigma \right) = 1$.
\item 
If $\omega_{\pi_1}\neq \omega_{\pi_2}$, $\omega_{\pi_1}|_{1+
\varpi \cO_2}=\omega_{\pi_2}|_{1+
\varpi \cO_2}$, then
$m \left(\pi^{\delta_5}_{N, \psi},\sigma \right) = 2$.
\item  
If $\omega_{\pi_1}=\omega_{\pi_2}$, then $m \left(\pi^{\delta_5}_{N, \psi},\sigma \right) = 3$.
\end{enumerate}
\end{proposition}

\begin{proof}

Using Mackey theory and Frobenius reciprocity, we get
\begin{align}\label{principal series computation with Hom}
\Hom_{\Delta \GL_2(\cO_2)\cdot N}(\pi^{\delta_5},\sigma\otimes \psi) 
& \cong \Hom_{\Delta \GL_2(\cO_2)\cdot N}\left(\Ind_{P_{\delta_5}}^{P}(\pi_1\otimes \pi_2)^{\delta_5}, \Ind_{\Delta \mathfrak{B}_2\cdot N}^{\Delta \GL_2(\cO_2)\cdot N}\omega_{\pi_1}\otimes \omega_{\pi_2}\otimes \psi\right)
\nonumber
\\
& \cong \Hom_{\Delta \mathfrak{B}_2\cdot N}\left(\Res_{\Delta\mathfrak{B}_2\cdot N}^{P}\Ind_{P_{\delta_5}}^{P}(\pi_1\otimes \pi_2)^{\delta_5}, \omega_{\pi_1}\otimes \omega_{\pi_2}\otimes \psi\right)
\nonumber
\\
& \cong \underset{\gamma\in \Delta\mathfrak{B}_2\cdot N\backslash P/P_{\delta_5}}{\bigoplus} W_{\gamma},
\end{align}
where  $W_{\gamma}:= \Hom_{\gamma P_{\delta_5}\gamma^{-1}\cap \Delta \mathfrak{B}_2\cdot N}\left((\pi_1\otimes\pi_2)^{\delta_5\gamma}, \omega_{\pi_1}\otimes \omega_{\pi_2}\otimes \psi\right)$.
Recall the subgroup $P_{\delta_5} \subseteq P$ given by 
{\small{\begin{equation} \label{subgroup p_delta5}
P_{\delta_5}  
= \left\{ Y := \begin{pmatrix}
    p_{11} & p_{12} & p_{13} & p_{14} \\
    0 & p_{22} & 0 & p_{24} \\
    0 & 0 & p_{33} & p_{34} \\
    0 & 0 & 0 & p_{44}
\end{pmatrix} ~{\Bigg{|}}~ \begin{matrix}
p_{ii} \in \cO_2^\times \text{ for }1\leq i\leq 4 , 
\\
\; p_{ij} \in \cO_2 \text{ for } i \neq j
\end{matrix}\right\}.
\end{equation}
}}
A set of representatives of $\Delta \mathfrak{B}_2\cdot N\backslash P/P_{\delta_5}$ is given by the set of matrices of the form $\Diag(g,h)$,
where $g$ varies over a set of representatives of  $\mathfrak{B}_2\backslash\GL_2(\cO_2)/\mathfrak{B}_2$ and $h$ varies over a set of representatives of  $\GL_2(\cO_2)/\mathfrak{B}_2$ as given in Lemma \ref{GL_2 double cosets}. 

Now we compute $W_{\gamma}$ appearing in \eqref{principal series computation with Hom} for every $\gamma\in \Delta\mathfrak{B}_2\cdot N\backslash P/P_{\delta_5}$.

\begin{enumerate}[label = (\alph*)]
\item 
For $\gamma=I_4$, we have $\gamma P_{\delta_5}\gamma^{-1}\cap \Delta \mathfrak{B}_2\cdot N = P_{\delta_5}\cap \Delta\mathfrak{B}_2\cdot N$.
Therefore, we have
$W_{\gamma} \neq 0$ if and only if $(\pi_1\otimes \pi_2)^{\delta_5}|_{P_{\delta_5}\cap \Delta \mathfrak{B}_2\cdot N}$ contains the character $(\omega_{\pi_1}\otimes \omega_{\pi_2}\otimes \psi)|_{P_{\delta_5}\cap \Delta \mathfrak{B}_2\cdot N}$. 
Recall that $Y \in P_{\delta_5}$ (see \eqref{subgroup p_delta5}).
Then $Y\in P_{\delta_5}\cap \Delta\mathfrak{B}_2\cdot N$ if and only if $p_{33}=p_{11},~p_{34}=p_{12},~p_{44}=p_{22}$.
For $Y\in P_{\delta_5}\cap \Delta \mathfrak{B}_2\cdot N$, we get
\begin{align}\label{equation 4.16 for w-model}
(\pi_1 \otimes \pi_2)^{\delta_5}(Y) 
& = \pi_1 \begin{pmatrix}
    p_{11} & p_{13} \\
    0 & p_{11}
\end{pmatrix} \otimes \pi_2\begin{pmatrix}
    p_{22} & p_{24} \\
    0 & p_{22}
\end{pmatrix} 
\nonumber
\\
& =
\omega_{\pi_1}(p_{11}) \cdot 
\pi_1 (u^+(p_{13}))
\otimes 
\omega_{\pi_2}(p_{22}) \cdot 
\pi_2 (u^+(p_{24})).
\end{align}
and
\begin{equation}\label{eq:4.12}
(\omega_{\pi_1} \otimes \omega_{\pi_2} \otimes \psi)(Y)
= \omega_{\pi_1}(p_{11}) \, \omega_{\pi_2}(p_{22}) \, \psi_0(p_{13} + p_{24}).
\end{equation}
Using Equation \eqref{equation 4.16 for w-model} and \eqref{eq:4.12}, together with Lemma \ref{tensor product on hom}, we get
$\dim\left( W_{\gamma} \right) = 1$.
\item For $\gamma=\Diag(w_0,w_0)$.
We have,
{\small{
$$
\gamma P_{\delta_5}\gamma^{-1}\cap \Delta \mathfrak{B}_2\cdot N =\left\{
Y^{w_0} : =\begin{pmatrix}
p_{22} & 0 & p_{24} &  0\\
0 &  p_{11} & p_{14} & p_{13}\\
0 & 0 & p_{22} & 0\\
0 & 0 & 0 & p_{11}
\end{pmatrix} ~{\Bigg{|}}~ \begin{matrix}
p_{11},p_{22}\in \cO_2^\times, 
\\
p_{ij}\in \cO_2 \text{ for $i\neq j$}
\end{matrix}
\right\}
$$
}}
It is easy to verify that $\dim (W_{\gamma}) =1$ if $\omega_{\pi_1}=\omega_{\pi_2}$ and 0 otherwise.

\item For $\gamma=\Diag(u^-(\varpi ),u^-(\varpi ))$.
The subgroup $\gamma P_{\delta_5}\gamma^{-1}\cap \Delta\mathfrak{B}_2\cdot N $ of $P$ is as follows
{\small{$$
\left\{
Y^{\varpi} : =\begin{pmatrix}

p_{11}-\varpi p_{12} & p_{12} & p_{13}-\varpi p_{14} &  p_{14}\\
0 &  p_{22}+\varpi p_{12} & \varpi(p_{13}-p_{24}) & p_{24}+\varpi p_{14}\\
0 & 0 & p_{11}-\varpi p_{12} & p_{12}\\
0 & 0 & 0 & p_{22}+\varpi p_{12}
\end{pmatrix} ~\middle |~ \begin{matrix}
p_{11},p_{22}\in \cO_2^\times, \\
p_{22}-p_{11}\in \varpi \cO_2,
\\
p_{ij}\in \cO_2 \text{ for $i\neq j$}
\end{matrix}
\right\}.
$$
}}
It is easy to verify that $\dim (W_{\gamma}) =1$ if $\omega_{\pi_1}|_{1+\varpi \cO_2}=\omega_{\pi_2}|_{1+\varpi \cO_2}$ and 0 otherwise.

\item Now, we prove that if $\gamma \notin \Gamma'=\left\{ \Diag(g,g) : g\in \left\{
I, w_0, u^-(\varpi) \right\} \right\}$, then $W_{\gamma}=0$.
It suffices to prove that, for $N_0 \subseteq P_{\delta_5}$ (as defined in \eqref{The subgoup N_0}),
\begin{equation}\label{redution of Hom compuattion}
\Hom_{\gamma N_0\gamma^{-1}\cap (\Delta \mathfrak{B}_2\cdot N)}\left((\pi_1\otimes \pi_2)^{\delta_5 \gamma},\, \omega_{\pi_1}\otimes \omega_{\pi_2}\otimes \psi\right) = 0.
\end{equation}
Note that for $\gamma=\Diag(g,h)\notin \Gamma'$,  
{\small{$\gamma N_0\gamma^{-1}\cap \Delta \mathfrak{B}_2\cdot N =\left\lbrace \tilde{Z}:=\begin{pmatrix}
    I_2 & Z\\
    0 & I_2
\end{pmatrix}~|~ Z=g \begin{pmatrix}
    0 & p_{14}\\
    0 & 0
\end{pmatrix}h^{-1}\right\rbrace$.}}
It can be checked easily that for all $\gamma\notin \Gamma'$, $(\pi_1\otimes \pi_2)^{\delta_5 \gamma}$ is trivial on $\gamma N_0\gamma^{-1}\cap \Delta \mathfrak{B}_2\cdot N$,
while  $\omega_{\pi_1}\otimes \omega_{\pi_2}\otimes\psi$ is non-trivial.
Hence, for $\gamma\notin \Gamma'$, $W_{\gamma}=0$.
\end{enumerate}
Now, the proposition follows from (a), (b), (c), (d) together with \eqref{principal series computation with Hom}.
\qedhere
\end{proof}

\begin{corollary}\label{delta_5 coset}
Let $\pi_1, \pi_2, \sigma$ be as in the previous proposition. Then $\pi^{\delta_5}_{N, \psi} \cong \sigma$.
\end{corollary}

\begin{proof}
By Lemma \ref{irreducibility of ps of gl2}, the number of irreducible components of $\sigma$ is either 1, 2 or 3.
Let $\chi_1 = \omega_{\pi_1}$ and $\chi_2 = \omega_{\pi_2}$ and $\chi = \chi_1 \otimes \chi_2$ a character of $\mathfrak{B}_2$ as defined in Section \ref{Section 2.7}.
\\
{\bf Case 1:} Suppose $\sigma$ is irreducible.
Since $\dim(\pi^{\delta_5}_{N,\psi}) = \dim(\sigma)$, the corollary follows from Proposition \ref{dimension of all prinicpal series in pi-delta_5-N,psi} $(i)$. 
\\
{\bf Case 2:}
Suppose $\sigma$ has two irreducible components.
By Lemma \ref{irreducibility of ps of gl2}, $\chi_1|_{1+ \varpi \cO_2} = \chi_2|_{1+\varpi \cO_2}$ but $\chi_1 \neq \chi_2$.
Using Proposition \ref{Proposition 2.1.6}, we have $\sigma = \sigma_1 \oplus \sigma_2$, where $\sigma_1 = \Ind_{\tilde{\mathfrak{B}}_2}^{\GL_2(\cO_2)}\tilde{\chi}$. 
By Mackey theory, we get
$$
\Hom_{\Delta \GL_2(\cO_2)\cdot N}\left(\pi^{\delta_5}, \sigma_1\otimes \psi\right) 
\cong 
\underset{\gamma\in \Delta\tilde{\mathfrak{B}}_2\cdot N\backslash P/P_{\delta_5}}{\bigoplus} 
W'_{\gamma},
$$
where $W'_{\gamma} := \Hom_{\gamma P_{\delta_5}\gamma^{-1}\cap \Delta \tilde{\mathfrak{B}}_2\cdot N}\left((\pi_1\otimes\pi_2)^{\delta_5\gamma}, \tilde{\chi}\otimes \psi\right)$.
Clearly, a set of distinct representatives for $\Delta\tilde{\mathfrak{B}}_2\cdot N\backslash P/P_{\delta_5}$ is 
$\left\{ \Diag(g,h) 
:g \in \{I_2, w_0\},~h\in \GL_2(\cO_2)/\mathfrak{B}_2\right\}$.
It can be checked that
$W'_{\gamma} =0$ for $\gamma\neq I_2$ and $\dim(W'_{I_2}) =1$ for which we leave the details. 
Thus, $m(\pi^{\delta_5}_{N, \psi}, \sigma_1) =1$.
From Proposition \ref{dimension of all prinicpal series in pi-delta_5-N,psi} $(ii)$, it follows that $m(\pi^{\delta_5}_{N, \psi}, \sigma_2) =1$.
Then $\sigma= \sigma_1 \oplus \sigma_2 \subseteq \pi^{\delta_5}_{N, \psi}$. 
Since $\dim(\pi^{\delta_5}_{N,\psi}) = \dim(\sigma)$, the corollary follows. 
\\
{\bf Case 3:} 
Suppose $\sigma$ has three irreducible components.
Then by Lemma \ref{irreducibility of ps of gl2}, $\chi_1=\chi_2$ and
by Proposition \ref{Proposition 2.1.6}, we have
$\sigma
\cong \sigma'_1 \oplus \sigma'_2 \oplus \sigma'_3$.
From Proposition~\ref{dimension of all prinicpal series in pi-delta_5-N,psi}~$(iii)$, it suffices to show that \( \sigma'_1 \) and \( \sigma'_2 \) each appear in \( \pi^{\delta_5}_{N,\psi} \) with multiplicity one. 
Using similar computations as in Case 2, it can be verified that
$m\left( \pi^{\delta_5}_{N,\psi}, \sigma'_1 \oplus \sigma'_2 \right) = 2$.
Therefore, it is enough to prove that \( \sigma'_1 =\chi_1\circ \det
\), appears in \( \pi^{\delta_5}_{N,\psi} \) with multiplicity one. This follows from Mackey theory, and we omit the details. 
This completes the proof of the corollary.
\qedhere
\end{proof}

\subsection{A description of $\pi^{\delta_6}_{N, \psi}$}
Since $\pi^{\delta_6} = \Ind_{\delta_6 P\delta_6^{-1}\cap P}^P(\pi_1\otimes\pi_2)^{\delta_6}$, 
using Mackey theory we have
\begin{align}\label{equation 4.11}
\pi^{\delta_6}_{N,\psi}=
\Hom_{N}(\pi^{\delta_6},\psi)
\cong \underset{\gamma\in \delta_6 P\delta_6^{-1}\cap P\backslash P/N}{\bigoplus} W_{\gamma},
\end{align}
where $W_{\gamma} : = \Hom_{ P\cap (\delta_6^{-1}N\delta_6)} \left(\pi_1\otimes\pi_2,\psi^{\gamma^{-1}\delta_6^{-1}} \right)$.
Now we write $P_{\delta_6} := \delta_6 P\delta_6^{-1}\cap P$.
The following lemma describes the double coset $P_{\delta_6}\backslash P/N$.

\begin{lemma}
\label{Several double cosets}
There is a bijection between the double cosets  $P_{\delta_6}\backslash P/N$ and $\bar{P}_{\delta_6}\cdot \bar{N}\backslash \bar{P}$. 
A set of distinct representatives of the cosets $\bar{P}_{\delta_6}\cdot \bar{N}\backslash \bar{P}$ is the following 
$$
\Gamma = 
\left\lbrace \Diag(g,h) : g \in \left\{ w_0, u^-(b) : b \in \F_{q} \right\}, 
h \in \left\{ dw_0,~ h(a,c)=\begin{pmatrix}
    1 & 0\\
    c & a
\end{pmatrix} :
c\in \F_q,~
a,d\in \F_q^\times \right\rbrace \right\rbrace.
$$
Moreover, $|\Gamma| = (q+1)(q^2-1)$.
\end{lemma}

\begin{proof}
Since $ N $ is a normal subgroup of $ P $, there exists a bijection between the double coset $ P_{\delta_6} \backslash P / N $ and $ (P_{\delta_6} \cdot N) \backslash P $. 
Since {\small{$P_{\delta_6} = \left\lbrace \begin{pmatrix}
    p_{11} & p_{12} & p_{13} & p_{14}\\
    -\varpi p_{24} & p_{22} & 0 & p_{24}\\
    0 & 0 & p_{33} & p_{34}\\
    0 & 0 & \varpi p_{13} & p_{44}
\end{pmatrix}:\begin{matrix}
 p_{11}-p_{44}\in \varpi \cO_2,
\\
p_{ij}\in \cO_2\text{ for }i\neq j,
\\
p_{ii}\in \cO_2^\times
\end{matrix}
\right\rbrace$}}, the subgroup $ P_{\delta_6} \cdot N $ contains $ (I_4 + \varpi M_4(\cO_2)) \cap P $, which gives a bijection between 
$(P_{\delta_6} \cdot N) \backslash P$ and $(\bar{P}_{\delta_6} \cdot \bar{N}) \backslash \bar{P}$.
Further
$\bar{P} / \bar{N} \cong \GL_2(\mathbb{F}_q) \times \GL_2(\mathbb{F}_q)$, we get 
$|(\bar{P}_{\delta_6} \cdot \bar{N}) \backslash \bar{P}| = (q+1)(q^2 - 1)$.
Now the lemma follows using Lemma \ref{GL_2 double cosets}.
\end{proof}

\begin{proposition}
\label{dimension of pi-delta-6}
If $\delta=\delta_6$, then 
$\dim(\pi^{\delta}_{N,\psi})=q(q^2-1)$.
\end{proposition}

\begin{proof}
We will use identity \eqref{equation 4.11} and above lemma to prove this proposition.
Note that 
\begin{equation} \label{expression for x}
\delta_6 P \delta_6^{-1} \cap N = \left\{ x= \begin{pmatrix} 
I_2 & X \\
0 & I_2 
\end{pmatrix} : X = \begin{pmatrix}  \varpi p_{13} & p_{14} \\ 0 & \varpi p_{24}  \end{pmatrix} ; p_{ij} \in \mathfrak{o}_2 \right\}.
\end{equation}
Let \( K := P \cap \delta_6^{-1} N \delta_6 \). 
Then \( K \) is given by
{\small{$$
\left\{\begin{pmatrix}
    K_1 & K_2\\
    0 & K_4
\end{pmatrix} \text{ with } K_1=\begin{pmatrix}
    1+\varpi p_{14} & \varpi p_{13}\\
    0 & 1
\end{pmatrix},~K_2=\begin{pmatrix}
        0 & p_{14}\\
        0 & 0
\end{pmatrix},~K_4=\begin{pmatrix}
        1 & \varpi p_{24}\\
        0 & 1-\varpi p_{14}
\end{pmatrix} : p_{13}, p_{14}, p_{24}\in \cO_2\right\}.
$$}
}
Then, for $x \in \delta_{6}P \delta_{6}^{-1}\cap N$ as in \eqref{expression for x}, we get
\begin{equation}\label{pi1 pi2 action}
(\pi_1\otimes \pi_2)(\delta_6^{-1} x \delta_6) =
\pi_1 \begin{pmatrix} 
1 + \varpi p_{14} & \varpi p_{13} \\ 0 & 1 \end{pmatrix}\otimes \pi_2\begin{pmatrix} 1 & \varpi p_{24} \\ 0 & 1 - \varpi p_{14} \end{pmatrix}.
\end{equation}
Recall that as a representation of $P$, 
$\pi_1 \otimes \pi_2 \cong \Ind_{T_1 T_2 N}^{P}(\tilde{\phi}_{B_1}\otimes \tilde{\phi}_{B_2}\otimes 1)$, where $T_{i} := I(\phi_{B_i})$ for $i=1,2$.
Using Mackey theory and \eqref{equation 4.11}, we get
\begin{align}\label{4.10}
\pi^{\delta_6}_{N,\psi}  \cong \underset{\gamma\in \Gamma}{\bigoplus}  W_{\gamma} \cong \underset{\gamma\in \Gamma}{\bigoplus} ~\underset{\beta\in \mathcal{B}}{\bigoplus} ~ W_{\gamma, \beta}.
\end{align}
where $W_{\gamma, \beta} = \Hom_{\beta^{-1}K\beta\cap (T_1 T_2 N)}\left(\tilde{\phi}_{B_1}\otimes  \tilde{\phi}_{B_2}\otimes 1,\psi^{\gamma^{-1}\delta_6^{-1}\beta^{-1}}\right)$ and $\mathcal{B}$ is a set of representatives for $K\backslash P/(T_1 T_2N)$.
Clearly, there is a bijection between $\mathcal{B}$ and the cosets $\GL_2(\F_q)/\F_{q^2}^\times\times \GL_2(\F_q)/\F_{q^2}^\times$ for which a set of representatives can be taken to be 
{\small{$$
\left\{
\beta(y,z;y',z') = \Diag \left( \begin{pmatrix}
        1 & y \\
        0 & z
\end{pmatrix}, \begin{pmatrix}
        1 & y' \\
        0 & z'
    \end{pmatrix} \right)
 :  z, z' \in \mathbb{F}_q^\times \text{ and } y, y' \in \mathbb{F}_q
\right\}.
$$
}}

Then for $\beta= \beta(y,z;y',z')$, we have $\beta^{-1}K\beta \subseteq (T_1 T_2 N)$. 
In fact,
{\small{$$
\beta^{-1}K\beta =\left\lbrace 
\begin{pmatrix}
    1+\varpi p_{14} & \varpi (p_{14}\tilde{y}+p_{13}\tilde{z}) & 0 & p_{14}\tilde{z}'\\
    0 & 1 & 0 & 0\\
    0 & 0 & 1 & \varpi (p_{14}\tilde{y}'+p_{24}\tilde{z}')\\
    0 & 0 & 0 & 1-\varpi p_{14}
\end{pmatrix} :  p_{13}, p_{14}, p_{24} \in \cO_2 \right\rbrace.
$$
}}
For $\beta = \beta(y,z;y',z') \in \mathcal{B}$, we have $\beta^{-1}K\beta = \beta^{-1} \delta_{6}^{-1} ( \delta_{6}P \delta_{6}^{-1} \cap N)\delta_{6} \beta$. 
For any $x\in \delta_{6}P \delta_{6}^{-1} \cap N$ as given in \eqref{expression for x}, we get
\begin{equation*}\label{action of psi^gamma for gamma in Gamma_1}
\psi^{\gamma^{-1} \delta_6^{-1}\beta^{-1}}(\beta^{-1} \delta_{6}^{-1} x \delta_{6} \beta)
= \psi^{\gamma^{-1}}(x).
\end{equation*}
Moreover, 
\begin{align} \label{character phi B_1 B_2}
(\tilde{\phi}_{B_1}\otimes  \tilde{\phi}_{B_2}\otimes 1) (\beta^{-1} \delta_{6}^{-1} x \delta_{6} \beta)
& =\phi_{B_1}\begin{pmatrix}
    1+\varpi p_{14} & \varpi(p_{14}\tilde{y}+p_{13}\tilde{z})\\
    0 & 1
\end{pmatrix}\phi_{B_2}\begin{pmatrix}
    1 & \varpi( p_{14}\tilde{y}'+p_{24}\tilde{z}')\\
    0 & 1-\varpi p_{14}
\end{pmatrix} \nonumber \\
& = \psi_0\left(((m_1-m_2)+n_1y+n_2y')\bar{p}_{14}+n_1z\bar{p}_{13}+n_2z'\bar{p}_{24}\right).
\end{align}
{\bf Case 1 :} Let
$ \Gamma_{1} := \left\{ \gamma(a,b,c) = \Diag( u^-(b) , h(a,c)) : a \in \F_{q}^{\times} \text{ and } b, c \in \F_q \right\} $.
Note that $\beta^{-1}K \beta = \beta^{-1} \delta_6^{-1} ( \delta_{6}P \delta_{6}^{-1} \cap N ) \delta_{6} \beta$.
For $\gamma = \gamma(a,b,c) \in \Gamma_{1}$ and $x\in \delta_{6}P\delta_{6}^{-1}\cap N$ as given in \eqref{expression for x} we get
\begin{equation}\label{action of psi^gamma for gamma in Gamma_1-1}
\psi^{\gamma^{-1} \delta_6^{-1}\beta^{-1}}(\beta^{-1} \delta_{6}^{-1} x \delta_{6} \beta)
= \psi^{\gamma^{-1}}(x)
= \psi_0\left( \frac{\tilde{b} - \tilde{c}}{\tilde{a}} p_{14} + \varpi p_{13} +  \frac{\varpi p_{24}}{\tilde{a}} \right)
\end{equation}
where $\tilde{a}, \tilde{b}, \tilde{c} \in \cO_2$ are the lifts of $a,b,c\in \F_q$ under the quotient map $\cO_2\twoheadrightarrow \cO_1\cong \F_q$.
By looking at the coefficient of $p_{14}$ in \eqref{character phi B_1 B_2} and \eqref{action of psi^gamma for gamma in Gamma_1} and using \eqref{4.10}, we get
\begin{align}
W_{\gamma,\beta}\neq 0 
&\implies  b=c.
\end{align}
Assume $b=c$ and write $\tilde{b} = \tilde{c} + \varpi k$ for some $k \in \cO_2$.
Thus, Equation \eqref{action of psi^gamma for gamma in Gamma_1} becomes
{\small{\begin{equation}\label{For gamma=Gamma_1}
\psi^{\gamma^{-1}}(x)= \psi_0\left(\dfrac{  \varpi k p_{14}}{\tilde{a}}+\varpi p_{13}+\dfrac{\varpi p_{24}}{\tilde{a}}\right).
\end{equation}
}}
Using \eqref{character phi B_1 B_2} and \eqref{For gamma=Gamma_1}, we get $W_{\gamma,\beta} \neq 0$ if and only if 
{\small{\begin{equation*}\psi_0\left(((m_1-m_2)+n_1y+n_2y')\bar{p}_{14}+n_1z\bar{p}_{13}+n_2z'\bar{p}_{24}\right)=\psi_0\left(\frac{\bar{k}\bar{p}_{14}}{a}+\bar{p}_{13}+\frac{\bar{p}_{24}}{a}\right)
\end{equation*}
}} 
for all $p_{12}, p_{13},p_{14}\in \cO_2$.
This gives 
{\small{$z=\dfrac{1}{n_1},~z'=\dfrac{1}{an_2},~y'=\dfrac{1}{n_2}\left(\dfrac{k}{a}+(m_2-m_1)-n_1y\right)$.}}
Thus, for a fixed $\gamma=\gamma(a,b,c) \in \Gamma_1$ with $b=c$, 
$\dim\left(W_{\gamma}\right) 
= \# \left\{ \beta\in \mathcal{B} : W_{\gamma,\beta} \neq 0 \right\}  = q$.
The number of matrices $\gamma=\gamma(a,b,c)\in \Gamma_1$ with $b=c$ is $q(q-1)$. 
Therefore, 
$\underset{\gamma\in \Gamma_1}{\sum}\dim (W_{\gamma} ) =q^2(q-1)$.
\\
{\bf Case 2 :}  Let
$ \Gamma_2 := \left\{ \gamma(d) = \Diag(
    w_0 , dw_0) : d \in \F_{q}^{\times} \right\}$.
For $\gamma= \gamma(d) \in \Gamma_2$ and $x \in \delta_6^{-1}K\delta_6$, we get
{\small{\begin{equation*}\label{character value of psi-gamma for Gamma_2}
\psi^{\gamma^{-1}}(x)  = \psi_0 \left( \frac{\varpi (p_{13} + p_{24}) }{\tilde{d}} \right).
\end{equation*}
}}
Using \eqref{character phi B_1 B_2}, we get $\Hom_{\beta^{-1}K\beta}\left(\tilde{\phi}_{B_1}\otimes  \tilde{\phi}_{B_2}\otimes 1,\psi^{\gamma^{-1}\delta_6^{-1}\beta^{-1}}\right) \neq 0$
if and only if 
{\small{\begin{equation*}
\psi_0\left(((m_1-m_2)+n_1y+n_2y')\bar{p}_{14}+n_1z\bar{p}_{13}+n_2z'\bar{p}_{24}\right)=\psi_0 \left( \frac{\bar{p}_{13} + \bar{p}_{24}}{d} \right)
\end{equation*}
}}
for all $p_{13}, p_{14}, p_{24}\in \cO_2$.
This gives 
{\small{$z=\dfrac{1}{n_1d},~z'=\dfrac{1}{n_2d},~y'=\dfrac{1}{n_2}((m_2-m_1)-n_1y)$.
}}
For any $\gamma=\gamma(d) \in \Gamma_2$, we get
$\dim\left(W_{\gamma}\right) = q$. 
Since $|\Gamma_2| = q-1$, 
$\underset{\gamma\in \Gamma_2}{\sum}
\dim \left(  W_{\gamma} \right) =q(q-1)$.
\\
{\bf Case 3 :} 
Let $\gamma \in \Gamma \smallsetminus (\Gamma_1 \cup \Gamma_2)$. Let $H=\left\{\begin{pmatrix}
    I_2 & X\\
    0 & I_2
\end{pmatrix} : X=\begin{pmatrix}
    0 & \varpi p_{14}\\
    0 & 0
\end{pmatrix} \right\} \subseteq \delta_{6}^{-1}K \delta_{6}$.
It can be checked that $\psi^{\gamma^{-1}}$ is non-trivial but $\tilde{\phi}_{B_1}\otimes  \tilde{\phi}_{B_2}\otimes 1$ is trivial on $\beta^{-1} \delta_{6}^{-1} H \delta_{6} \beta \subseteq \beta^{-1} K \beta$.
This gives $W_{\gamma, \beta}=0$ for all $\beta$.  
\\ 
Using \eqref{4.10}, together with Case 1, Case 2 and Case 3, we get 
\[
\dim(\pi^{\delta_6}_{N, \psi}) = \sum_{\gamma \in \Gamma} \dim (W_{\gamma}) = q^2(q-1) + q(q-1) = q(q^2-1).
\qedhere
\]
\end{proof}

\begin{lemma}\label{reduction of computation by saying other gamma=I}
For $\sigma\in \Irr(\GL_2(\cO_2))$,
$$
\Hom_{\Delta \GL_2(\cO_2)\cdot N}\left(\pi^{\delta_6}, \sigma\otimes\psi\right) \cong \Hom_{P_{\delta_6}\cap \Delta \GL_2(\cO_2)\cdot N}\left((\pi_1\otimes \pi_2)^{\delta_6},(\sigma\otimes\psi )\right).
$$
\end{lemma}
\begin{proof}
Using Mackey theory, we get
\begin{align}\label{equation ijn terms of sigma sec-4, 4.26}
\Hom_{\Delta \GL_2(\cO_2)\cdot N}\left(\pi^{\delta_6}, \sigma\otimes \psi\right)  
& \cong \underset{\gamma \in \Gamma}{\bigoplus}~\Hom_{P_{\delta_6}^{\gamma}}\left((\pi_1\otimes \pi_2)^{\delta_6},(\sigma\otimes\psi )^{\gamma^{-1}}\right)
\end{align}
where $\Gamma$ is a set of representatives for $\Delta \GL_2(\cO_2)\cdot N\backslash P/P_{\delta_6}$ and $P_{\delta_6}^{\gamma} = P_{\delta_6}\cap \gamma^{-1}(\Delta \GL_2(\cO_2)\cdot N)\gamma$. 
There is a bijection between set $\Gamma$ and $\Delta \GL_2(\F_q) \cdot \bar{N}\backslash \bar{P}/\bar{P}_{\delta_6}$ for which a set of representatives can be written using Lemma \ref{GL_2 double cosets}.
Recall the subgroup $N_0$ as defined in \eqref{The subgoup N_0}.
Note that $N_0\subseteq P_{\delta_6}^{\gamma}$ and
$(\pi_1 \otimes \pi_2)^{\delta_6}$ is trivial on $N_0$.
Moreover, it can be verified that if $\gamma$ represents a non-trivial double coset, then $(\sigma \otimes \psi)^{\gamma^{-1}}$ is non-trivial on $N_0$. 
Hence, using \eqref{equation ijn terms of sigma sec-4, 4.26}, the lemma follows.
\end{proof}

Recall that
$\pi_1 \otimes \pi_2 \cong \Ind_{T_1T_2N}^{P}(\tilde{\phi}_{B_1}\otimes \tilde{\phi}_{B_2}\otimes 1)$.
Using Mackey theory and using Lemma \eqref{reduction of computation by saying other gamma=I}, we get
\begin{equation}\label{equation 4.24-1}
\Hom_{\Delta \GL_2(\cO_2)\cdot N}\left(\pi^{\delta_6}, \sigma\otimes\psi\right)
\cong \underset{\gamma \in \Gamma_{I}}{\bigoplus} \Hom_{\gamma^{-1}K_{I}\gamma\cap T_1T_2N}\left(\tilde{\phi}_{B_1}\otimes\tilde{\phi}_{B_2}\otimes 1,(\sigma\otimes\psi )^{\delta_6^{-1}\gamma^{-1}}\right),
\end{equation}
where $K_{I}=\delta_6^{-1}\left(P_{\delta_6}\cap (\Delta\GL_2(\cO_2)\cdot N)\right)\delta_6$ is given by 
{\small{$$
\left\{
\begin{pmatrix}
    p_{11}+\varpi p_{14} & p_{13} & p_{12} & p_{14}\\
   \varpi p_{12} & p_{11} & 0 & p_{12}\\
    0 & 0 & p_{22} & p_{24}\\
    0 & 0 & -\varpi p_{12} & p_{22}-\varpi p_{14}
\end{pmatrix}:\begin{matrix}
p_{22}-p_{11}\in \varpi \cO_2, p_{11}\in \cO_2^\times,
\\
p_{12},p_{13},p_{14},p_{24}\in \cO_2,\\
p_{13}+p_{24}\in \varpi \cO_2
\end{matrix}
\right\}.
$$
}}
and $\Gamma_{I}$ is a set of representatives of $K_{I}\backslash P/T_1T_2N$
which is in bijection with the following set
% . The set $\Gamma_I$ can be taken to be 
{\small{\begin{equation}\label{Gamma_I}
\left\{
\gamma(y,z;y',z') = \Diag\left(\begin{pmatrix}
        1 & y \\
        0 & z
\end{pmatrix}, \begin{pmatrix}
        1 & y' \\
        0 & z'
\end{pmatrix}\right) :
    z, z' \in \mathbb{F}_q^\times,~ y, y' \in \mathbb{F}_q
\right\}.
\end{equation}
}}

\begin{theorem}\label{main theorem for 2-2 parabolic subgroup}
For $i=1,2$, let $\pi_i = \Ind_{I(\phi_{B_i})}^{\GL_2(\cO_2)}  \tilde{\phi}_{B_i}$ be strongly cuspidal representations of $\GL_2(\cO_2)$ with central characters $\omega_{\pi_1}$ and $\omega_{\pi_2}$ respectively.
Assume that $\tr(B_1) \neq \tr(B_2)$.
Let $B= \diag(m,n) \in M_2(\F_q)$ with $m \neq n$. 
Let $\tilde{\phi}_{B}$ be an extension to $I(\phi_{B})$ of the character $\omega_{\pi} \phi_{B}$ of $Z \cdot J^1_2$ and $\sigma = \Ind_{I(\phi_{B })}^{\GL_2(\cO_2)} \tilde{\phi}_{B}$.
Then $m (\pi^{\delta_6}_{N, \psi}, \sigma)=1$ if and only if $(m,n) \in \{ (\tr(B_1), \tr(B_2)), ~~~ (\tr(B_2), \tr(B_1)) \}$. 
\end{theorem}

\begin{proof}
Using Mackey theory and using \eqref{equation 4.24-1}, we get
{\small{\begin{align}\label{4.38 equation for the theorem 4.1.9}
\Hom_{\Delta \GL_2(\cO_2)\cdot N}\left(\pi^{\delta_6}, \sigma\otimes\psi\right)
&  \cong   \underset{\gamma \in \Gamma_{I}}{\bigoplus} \Hom_{K_{I}^{\gamma}}\left((\tilde{\phi}_{B_1}\otimes\tilde{\phi}_{B_2}\otimes 1)^{\gamma\delta_6},\Res_{K_{I}^{\gamma}}^{\Delta \GL_2(\cO_2)\cdot N}\Ind_{\Delta I(\phi_B)\cdot N}^{\Delta \GL_2(\cO_2)\cdot N}(\tilde{\phi}_B\otimes\psi) \right)
\nonumber
\\
& \cong \underset{\gamma \in \Gamma_I}{\bigoplus}\underset{\lambda \in \Lambda_{\gamma}}{\bigoplus}\Hom_{\lambda^{-1}K_{I}^{\gamma}\lambda\cap \Delta I(\phi_B)\cdot N}\left((\tilde{\phi}_{B_1}\otimes \tilde{\phi}_{B_2}\otimes 1)^{\gamma \delta_6\lambda^{-1}},(\tilde{\phi}_B\otimes \psi)\right),
\end{align}
}}
where $\gamma=\gamma(y,z;y',z')\in \Gamma_{I}$ (see \eqref{Gamma_I}),  $K_{I}^{\gamma}=\delta_6[K_I\cap \gamma(T_1T_2N)\gamma^{-1}]\delta_6^{-1}$ and $\Lambda_{\gamma}$ is a set of representatives of the double coset $K_{I}^{\gamma}\backslash \Delta \GL_2(\cO_2)\cdot N/\Delta I(\phi_B)\cdot N$.
Then $\Lambda_{\gamma}$ is in bijection with 
$$
\left\lbrace \Diag( u^-(x), u^-(x)): x \in \F_q\right\rbrace  \bigcup~ \{W_0\},
$$
where $W_{0} = \Diag( w_0, w_0)$.
Let $W_{\gamma, \lambda}:=\Hom_{\lambda^{-1}K_{I}^{\gamma}\lambda\cap \Delta I(\phi_B)\cdot N}\left((\tilde{\phi}_{B_1}\otimes \tilde{\phi}_{B_2}\otimes 1)^{\gamma \delta_6\lambda^{-1}},(\tilde{\phi}_B\otimes \psi)\right)$. 
Using \ref{4.38 equation for the theorem 4.1.9}, we get 
\begin{equation}\label{main equation for hom in Theprem 3.12}
\Hom_{\GL_2(\cO_2)}(\pi^{\delta_6}_{N,\psi}, \sigma)\cong \underset{\gamma \in \Gamma_I}{\bigoplus}~\underset{\lambda\in \Lambda_{\gamma}}{\bigoplus}~W_{\gamma, \lambda}.
\end{equation} 
Now the theorem follows from the following three claims.\\
{\bf Claim 1:}
For any $\gamma\in \Gamma_{I}$ and  $\lambda \in \Lambda_{\gamma} \smallsetminus \{ I_4, W_{0} \}$,  
$W_{\gamma, \lambda}= 0$.
\\
{\bf Claim 2:} 
For $\lambda=I_4$, 
$\underset{\gamma\in \Gamma_{I}}{\sum}\dim\left(W_{\gamma, \lambda}\right)=1$ if $B=\diag(\tr(B_1), \tr(B_2))$.
\\
{\bf Claim 3:} For $\lambda=W_0$, 
$\underset{\gamma\in \Gamma_{I}}{\sum}\dim\left(W_{\gamma, \lambda}\right)=1$ if $B=\diag(\tr(B_2), \tr(B_1))$.
\begin{proof}[Proof of Claim 1:]  \let\qed\relax
Let $\lambda\in \Lambda_{\gamma}\smallsetminus \{ W_0 \}$, i.e., $\lambda=\Diag(u^-(x),u^-(x))$, where $\tilde{x}\in \cO_2$ is a lift of $x\in \F_q$. 
Then the subgroup $\lambda^{-1} K_{I}^{\gamma}\lambda \cap \Delta I(\phi_B)\cdot N$ is given by 
{\footnotesize{$$
\left\lbrace 
Z : =\begin{pmatrix}
    p_{11}+\varpi p_{12}\tilde{x} & \varpi p_{12} & \varpi p_{13}+p_{14}\tilde{x} & p_{14}\\
    (p_{22}-p_{11})\tilde{x}-\varpi p_{12}\tilde{x}^2 & p_{22}-\varpi p_{12}\tilde{x} & \varpi(p_{24}\tilde{x}-p_{13}\tilde{x})-p_{14}\tilde{x}^2 & \varpi p_{24}-p_{14}\tilde{x}\\
    0 & 0 & p_{11}+\varpi p_{12}\tilde{x} & \varpi p_{12}\\
    0 & 0  & (p_{22}-p_{11})\tilde{x}-\varpi p_{12}\tilde{x}^2 & p_{22}-\varpi p_{12}\tilde{x}
\end{pmatrix} ~~ \middle| ~~  \begin{matrix}
p_{11}\in \cO_2^\times, 
\\
p_{ij}\in \cO_2 \text{ for $i\neq j$},
\\
p_{22}-p_{11}\in \varpi \cO_2
\end{matrix}
\right\rbrace.
$$
}}
For $\gamma=\gamma(y,z;y',z')\in \Gamma_I$ as defined in \eqref{Gamma_I}, we have
{\small{\begin{equation}\label{4.28-1-1}
(\tilde{\phi}_{B_1}\otimes \tilde{\phi}_{B_2}\otimes 1)^{\gamma \delta_6\lambda^{-1}}(Z)=\tilde{\phi}_{B_1}\begin{pmatrix}
    p_{11}+\varpi p_{14} & \varpi (p_{14}\tilde{y}+p_{13}\tilde{z})\\
    0 & p_{11}
\end{pmatrix}\tilde{\phi}_{B_2}\begin{pmatrix}
    p_{22} & \varpi (p_{14}\tilde{y}'+p_{24}\tilde{z}')\\
    0 & p_{22}-\varpi p_{14}
\end{pmatrix}
\end{equation}}}
and  
\begin{equation}\label{4.46 equation for RHS comparsion for characters}
(\tilde{\phi}_B\otimes \psi)(Z)=\tilde{\phi}_B\begin{pmatrix}
    p_{11}+\varpi p_{12}\tilde{x} & \varpi p_{12}\\
    (p_{22}-p_{11})\tilde{x}-\varpi p_{12}\tilde{x}^2 & p_{22}-\varpi p_{12}\tilde{x}
\end{pmatrix}\psi_0(\varpi p_{13}+\varpi p_{24}).
\end{equation}
Let $K_0^{(\gamma,\lambda)}$ be the subgroup of $\lambda^{-1} K_{I}^{\gamma}\lambda \cap \Delta I(\phi_B)\cdot N$ defined as follows: 
{\small{$$
K_0^{(\gamma,\lambda)}: = \left\lbrace Z(p_{12}):=\Diag\left(\begin{pmatrix}
    1+\varpi p_{12}\tilde{x} & \varpi p_{12} \\
    -\varpi p_{12}\tilde{x}^2 & 1-\varpi p_{12}\tilde{x}
    \end{pmatrix}, \begin{pmatrix}
   1+\varpi p_{12}\tilde{x} & \varpi p_{12} \\
    -\varpi p_{12}\tilde{x}^2 & 1-\varpi p_{12}\tilde{x}
\end{pmatrix}\right)~{\Big{|}}~ p_{12}\in \cO_2~\right\rbrace.
$$
}}
Clearly, using \eqref{4.28-1-1} the character $(\tilde{\phi}_{B_1}\otimes \tilde{\phi}_{B_2}\otimes 1)^{\gamma \delta_6\lambda^{-1}}$  is trivial on $K_0^{(\gamma,\lambda)}$.
Using \eqref{4.46 equation for RHS comparsion for characters} 
we get
$$
(\tilde{\phi}_B\otimes \psi) (Z(p_{12})) = \psi_0\left((m-n)x\bar{p}_{12}\right),
$$
which is a non-trivial character of $K_0^{(\gamma,\lambda)}$ if and only if $x\neq 0$ and $m \neq n$. 
Therefore, for $\lambda \neq I_4$, $W_{\gamma, \lambda}=0$.
\end{proof}

\begin{proof}[Proof of Claim 2] \let\qed\relax
Let $\lambda=I_4$, i.e. $x=0$ in Claim 1. 
The computations in the proof of Claim 1 are also valid for $x=0$.
Then $W_{\gamma,\lambda}\neq 0$ if and only if $(\tilde{\phi}_{B_1}\otimes \tilde{\phi}_{B_2}\otimes 1)^{\gamma \delta_6\lambda^{-1}} = \tilde{\phi}_B\otimes \psi$ on $\lambda^{-1}K_{I}^{\gamma}\lambda\cap \Delta I(\phi_B)\cdot N$.
Since $B_i=\begin{pmatrix}
    m_i & n_i\alpha\\
    n_i & m_i
\end{pmatrix}\in \F_{q^2}\smallsetminus \F_q$ for $i=1,2$, using \eqref{4.28-1-1} and \eqref{4.46 equation for RHS comparsion for characters}, we get $W_{\gamma,\lambda}\neq 0$ if and only if 
$$
z=\frac{1}{n_1},~z'=\frac{1}{n_2},~y'=\frac{m_2-m_1-n_1y}{n_2},~m=\tr(B_1),~n=\tr(B_2).
$$

Therefore $W_{\gamma, \lambda} \neq 0 \implies \gamma \in \left\{ \gamma\left(y, \frac{1}{n_1};\frac{m_2 - m_1 - n_1 y}{n_2}, \frac{1}{n_2} \right) : y \in \mathbb{F}_q \right\}$
and it can be checked that all the elements in this set represent the same double coset in $\Gamma_{I}$.
Therefore, if $\lambda =I_4$ and $B =\diag(\tr(B_1), \tr(B_2))$, then
$\underset{\gamma\in \Gamma_{I}}{\sum}\dim\left(W_{\gamma, \lambda}\right)=1$.
\end{proof}

\begin{proof}[Proof of Claim 3] 
Let $\lambda=W_0$.
Then 
{\small{$$
\lambda^{-1} K_{I}^{\gamma}\lambda \cap \Delta I(\phi_B)\cdot N
=
\left\lbrace Z' :=\begin{pmatrix}
    p_{22} & 0 & \varpi p_{24} & 0\\
    \varpi p_{12} & p_{11} & p_{14} & \varpi p_{13}\\
    0 & 0 & p_{22} & 0\\
    0 & 0  & \varpi p_{12} & p_{11}
\end{pmatrix}
~~ \middle| ~~
\begin{matrix}
p_{22}-p_{11}\in \varpi \cO_2\\
p_{11},p_{22}\in \cO_2^\times,\\
p_{ij}\in \cO_2 \text{ for $i\neq j$}
\end{matrix}\right\rbrace.
$$
}}
Let $\gamma=\gamma(y,z;y',z')\in \Gamma_I$, as defined in \eqref{Gamma_I}. 
Then we get
{\small{\begin{equation}\label{4.28-2}
(\tilde{\phi}_{B_1} \otimes \tilde{\phi}_{B_2}\otimes 1)^{\gamma \delta_6 \lambda^{-1}} (Z')=\tilde{\phi}_{B_1}\begin{pmatrix}
    p_{11}+\varpi p_{14} & \varpi (p_{14}\tilde{y}+p_{13}\tilde{z})\\
    0 & p_{11}
\end{pmatrix}\tilde{\phi}_{B_2}\begin{pmatrix}
    p_{22} & \varpi (p_{14}\tilde{y}'+p_{24}\tilde{z}')\\
    0 & p_{22}-\varpi p_{14}
\end{pmatrix}
\end{equation}}}
and
\begin{equation}\label{4.32}
(\tilde{\phi}_{B}\otimes\psi)(Z')=\tilde{\phi}_B\begin{pmatrix}
    p_{22} & 0\\
    \varpi p_{12} & p_{11}
\end{pmatrix}\psi_0(\varpi p_{13}+\varpi p_{24}).
\end{equation}
Using \eqref{4.28-2} and \eqref{4.32}, we get that  $W_{\gamma, \lambda}\neq 0$ if and only if 
$$
z=\frac{1}{n_1},~z'=\frac{1}{n_2},~y'=\frac{m_2-m_1-n_1y}{n_2},~m=\tr(B_2),~n=\tr(B_1).
$$ 
Following the arguments similar to the case for $\lambda=I_4$, we get that, if 
$B =\diag(\tr(B_2), \tr(B_1))$, then 
$\underset{\gamma\in \Gamma_{I}}{\sum}\dim\left(W_{\gamma, \lambda}\right)=1$.
\end{proof}
\let\qed\relax
\end{proof}

\begin{corollary}\label{delta_6 coset}
Assume $\tr(B_1) \neq \tr(B_2)$. Let $B= \diag(\tr(B_1), \tr(B_2))$ or $\diag( \tr(B_2), \tr(B_1) )$.
Then 
\[
\pi^{\delta_6}_{N,\psi} \cong \Ind_{Z \cdot J^1_2}^{\GL_2(\cO_2)} (\omega_{\pi} \phi_B).
\]
\end{corollary}
\begin{proof}
Note that $\Ind_{Z \cdot J^1_2}^{I(\phi_B)} (\omega_{\pi} \phi_{B})$ is direct sum of $|I(\phi_B)/Z \cdot J^1_2 | = q-1$ distinct characters of $I(\phi_B)$.
By Theorem \ref{main theorem for 2-2 parabolic subgroup}, if $\tilde{\phi}_{B}$ appears in $\Ind_{Z \cdot J^1_2}^{I(\phi_B)} (\omega_{\pi} \phi_{B})$ then $\Ind_{I(\phi_B)}^{\GL_2(\cO_2)}\tilde{\phi}_B$ appears in $\pi^{\delta_6}_{N, \psi}$ with multiplicity one.
Therefore, we get
\begin{equation}
\Ind_{Z \cdot J^1_2}^{\GL_2(\cO_2)} (\omega_{\pi} \phi_B) \cong \Ind_{I(\phi_B)}^{\GL_2(\cO_2)} \left( \Ind_{Z \cdot J^1_2}^{I(\phi_B)} (\omega_{\pi} \phi_B) \right) \subseteq \pi^{\delta_6}_{N,\psi}.
\end{equation}
Using Lemma \ref{dimension of pi-delta-6}, we get $\dim (\pi^{\delta_6}_{N,\psi}) = \dim \left(\Ind_{Z \cdot J^1_2}^{\GL_2(\cO_2)} (\omega_{\pi} \phi_B)\right)$, and the lemma follows. 
\end{proof}

%We summarize the results discussed in this Section in the theorem below.

% \begin{theorem} \label{induction from P}
% Let $\pi_1 = \Ind_{I(\phi_{B_1})}^{\GL_2(\cO_2)}  (\tilde{\phi}_{B_1})$ and $\pi_2 =  \Ind_{I(\phi_{B_2})}^{\GL_2(\cO_2)}  (\tilde{\phi}_{B_2})$ be strongly cuspidal representations of $\GL_2(\cO_2)$ with central characters $\omega_{\pi_1}$ and $\omega_{\pi_2}$, respectively.
% Let $\pi=\Ind_{P}^{\GL_4(\cO_2)}(\pi_1\otimes \pi_2)$.
% Assume that $\tr(B_1) \neq \tr(B_2)$ and $B  = \diag (\tr(B_1), \tr(B_2))$.
% Then 
% \[
% \pi_{N,\psi} \cong 
% (\pi_1\otimes \pi_2) ~~~~ \bigoplus ~~~~ \Ind_{\mathfrak{B}_2}^{\GL_2(\cO_2)} (\omega_{\pi_1} \otimes \omega_{\pi_2}) 
% ~~~~ \bigoplus ~~~~ \Ind_{Z \cdot J^1_2}^{\GL_2(\cO_2)} (\omega_{\pi} \phi_B).
% \]
% \end{theorem}
Now we prove the Theorem \ref{induction from P-intro} mentioned in the introduction.
\begin{proof}[Proof of Theorem \ref{induction from P-intro}]
Note that $\pi_{N, \psi} \cong \oplus_{i=1}^{6} \pi^{\delta_i}_{N, \psi}$.
By Proposition \ref{deltas with 0 WM}, $\pi^{\delta_i}_{N, \psi} =0$ for $i=1,2,3$. 
Therefore 
\[
\pi_{N, \psi} \cong \pi^{\delta_4}_{N, \psi} \oplus  \pi^{\delta_5}_{N, \psi} \oplus \pi^{\delta_6}_{N, \psi}.
\]
By Corollary \ref{delta_4 coset}, $\pi^{\delta_4}_{N, \psi} \cong \pi_1 \otimes \pi_2$; 
by Corollary \ref{delta_5 coset}, $\pi^{\delta_5}_{N, \psi} \cong \Ind_{\mathfrak{B}_2}^{\GL_2(\cO_2)} (\omega_{\pi_1} \otimes \omega_{\pi_2})$ and by Corollary \ref{delta_6 coset}, $\pi^{\delta_6}_{N, \psi} \cong \Ind_{Z \cdot J^1_2}^{\GL_2(\cO_2)} (\omega_{\pi} \phi_B)$.
This proves Theorem \ref{induction from P-intro}.
%Then the theorem follows from Corollary \ref{delta_4 coset}, Corollary \ref{delta_5 coset} and Corollary \ref{delta_6 coset}.
\end{proof}

\section{$\pi_{N,\psi}$ for $\pi=\Ind_{Q}^{\GL_4(\cO_2)}(\rho\otimes \chi)$}\label{Sec : 5}

Let $\chi$ be a character of $\cO_2^\times$ and $\rho$ an irreducible strongly cuspidal representation of $\GL_3(\cO_2)$.
By Theorem \ref{Construction of regualr reprresentation}, $\rho \cong \Ind_{I(\phi_{C})}^{\GL_3(\cO_2)} \tilde{\phi}_{C}$ for some regular elliptic element $C \in M_3(\F_q)$.
Consider $\rho \otimes \chi$ as a representation of $Q$ via $Q \rightarrow Q/U \cong \GL_3(\cO_2) \times \GL_1(\cO_2)$. 
In this section, we describe the degenerate Whittaker space $\pi_{N, \psi}$ for the induced representation $\pi=\Ind_{Q}^{\GL_4(\cO_2)}(\rho\otimes \chi)$.
In order to compute $\pi_{N,\psi}$, we first restrict $\pi$ to the subgroup $P$ and use Mackey theory to understand the components of $\pi$ as a representation of $P$.
We get
\[
\pi|_{P}=(\Ind_{Q}^{\GL_4(\cO_2)}(\rho\otimes\chi))|_{P}=\underset{\delta\in P\backslash \GL_4(\cO_2)/Q}{\bigoplus} \pi^{\delta}
\]
where $\pi^{\delta}=\Ind_{\delta Q\delta^{-1}\cap P}^{P}(\rho\otimes\chi)^{\delta}$.
Therefore,
\[
\pi_{N,\psi}=\underset{\delta \in P\backslash \GL_4(\cO_2)/Q}{\bigoplus}\pi^{\delta}_{N,\psi}.
\]

\subsection{A description of  $P\backslash \GL_4(\cO_2)/Q$}
The following lemma follows from Proposition \ref{Double cosets proposition over finite field}, for which we skip the proof.
\begin{lemma}\label{Double coset for finite field}
A set of distinct representatives for $\bar{P}\backslash \GL_4(\F_q)/\bar{Q}$ is given by $\{I_4,(24)\}$, where $(24)$ is the permutation matrix obtained by interchanging second and forth rows of $I_4$.
\end{lemma}

\begin{corollary} 
A set of distinct representatives for $P\backslash \GL_4(\cO_2)/Q$ is
{\small{\begin{center}
$\left\lbrace I_4,~I^{\varpi}=\begin{pmatrix}
   I_2 & 0\\
   \begin{pmatrix}
       0 & 0\\
       \varpi & 0
   \end{pmatrix} & I_2
\end{pmatrix},~(24)\right\rbrace$.
\end{center}}}
\end{corollary}

\subsection{$\pi^{\delta}_{N, \psi} =0$ for $\delta =I, I^{\varpi}$}
\begin{lemma}\label{important lemma for reduction of Hom}
For any $\delta \in \{ I_4, I^{\varpi}, (24) \}$, we have 
\[
\pi^{\delta}_{N,\psi} \cong 
\bigoplus_{\beta \in N \backslash P / \delta Q \delta^{-1} \cap P} \Hom_{ (\delta Q \delta^{-1} \cap P)  \cap N } \left( (\rho \otimes \chi)^{\delta}, \psi^{\beta^{-1}} \right)
\]
\end{lemma}
\begin{proof}
Since $\pi^{\delta}=\Ind_{\delta Q\delta^{-1}\cap P}^{P}(\rho\otimes\chi)^{\delta}$, the lemma follows from Mackey theory.    
\end{proof}

\begin{proposition}\label{pi^I, pi_I_w have zero Whittaker models}
For $\delta\in \{I_4, I^{\varpi}\}$, $\pi^{\delta}_{N,\psi}=0$.
\end{proposition}
\begin{proof}
For the given $\delta$ and for any $\beta \in P$, we have $\beta (\delta Q \delta^{-1} \cap P) \beta^{-1} \cap N =N$. 
Let 
\[
J := \left\{ \begin{pmatrix}
    I_2 & X \\
    0 & I_2
\end{pmatrix} : X= \begin{pmatrix}
    0 & \varpi y\\
    0 & \varpi w
\end{pmatrix} \text{ with } y, w \in \cO_2 
\right\} \subseteq N.
\]
It can be checked that $(\rho \otimes \chi)^{\delta}$ is trivial on $J$, but  $\psi^{\beta^{-1}}$ is non-trivial.
Now the proposition follows from Lemma \ref{important lemma for reduction of Hom}.
\end{proof}

\begin{remark}
\noindent
It is important to note that the above proposition holds for any representation $\rho$ of 
$\GL_3(\cO_2)$.
\end{remark}

\subsection{A description of $\pi^{(24)}_{N, \psi}$}\label{pi_N,psi for last double coset for Q}
We first obtain the dimension of $\pi^{(24)}_{N, \psi}$.

\begin{theorem}\label{dim of pi^24_N,psi}
Let $\delta=(24)$. 
Then $\dim(\pi^{\delta}_{N,\psi})=q^2(q^2-1)$.
\end{theorem}

\begin{proof}
Let $P_{(24)}=(24)Q(24)\cap P$ and $L_0=(24)(P_{(24)}\cap N)(24)$.
Since $\rho = \Ind_{I(\phi_C)}^{\GL_3(\cO_2)} (\tilde{\phi}_C)$, we get
 $\rho\otimes \chi=\Ind_{S\cdot U\cdot \cO_2^\times}^{Q}(\tilde{\phi}_C\otimes 1\otimes\chi)$, where $S:= I(\phi_{C})$.
Using Mackey theory and Lemma \ref{important lemma for reduction of Hom}, we get
\begin{align}\label{dimension of pi24}
\pi^{(24)}_{N,\psi} 
& \cong \underset{\beta \in N\backslash P/P_{(24)}}{\bigoplus}~\underset{\gamma\in L_0\backslash Q/S\cdot U\cdot \cO_2^\times}{\bigoplus}W_{\gamma, \beta}, 
\end{align}
where $W_{\gamma, \beta}:=\Hom_{\gamma^{-1}L_0\gamma\cap(S\cdot U\cdot \cO_2^\times)}\left((\tilde{\phi}_C\otimes 1\otimes\chi),\psi^{\beta^{-1} (24)\gamma^{-1}}\right)$.
Since $N$ is normal in $P$, there is a bijection between $N\backslash P/P_{(24)}$ and $P/P_{(24)}\cdot N$. 
Then, we get a bijection between  $P/P_{(24)}\cdot N$ and $\GL_2(\cO_2)/\mathfrak{B}_2$.
Thus, a set of representatives of $N\backslash P/P_{(24)}$ is given by
{\small{$$
\left\lbrace
\beta_1(x)=
\begin{pmatrix}
\begin{pmatrix}
    1 & 0\\
    x & 1
\end{pmatrix} & 0\\
0 & I_2
\end{pmatrix}
:x\in \cO_2\right\rbrace\bigcup \left\lbrace\beta_2(y)=\begin{pmatrix}
\begin{pmatrix}
    \varpi y & 1\\
    1 & 0
\end{pmatrix} & 0\\
0 & I_2
\end{pmatrix} :y\in \cO_2
\right\rbrace.
$$
}}
Now, we compute $L_0\backslash Q/S\cdot U\cdot \cO_2^\times$.
More explicitly,
$$
L_0 = \left\lbrace
\begin{pmatrix}
1 & \mathbf{v} \\
0 & I_3
\end{pmatrix}
: \mathbf{v} = \begin{pmatrix} p_{14} &  p_{13} & 0 \end{pmatrix}, p_{14}, p_{13} \in \cO_2
\right\rbrace.
$$
Since $ (I_4 + \varpi M_4(\cO_2)) \cap Q\subseteq S\cdot U\cdot \cO_2^\times$  and $\GL_3(\F_q)= P_{1,2}\cdot \F_{q^3}^\times$ (see Lemma \ref{double cosets P-1-n-2_GL_n-F_{q^n}}), there is a bijection between $L_0\backslash Q/S\cdot U\cdot \cO_2^\times$ and $\{ \gamma(g)= \Diag( 1, g, 1) \in \GL_4(\F_q) : g \in \GL_2(\F_q) \}$.
Now, we compute $W_{\gamma, \beta}$ for all $\beta,\gamma$.
\\
Consider $\gamma=\gamma(g)$ for $g=\begin{pmatrix}
   g_1 & g_2\\
    g_3 & g_4
 \end{pmatrix}\in \GL_2(\F_q)$.
Then 
$$
L_0^{\gamma}
:=\gamma^{-1}L_0\gamma \cap(S\cdot U\cdot \cO_2^\times)
=\left\lbrace Y :
=
\begin{pmatrix}
1 & \mathbf{v}^T \\
0 & I_3
\end{pmatrix}
: \mathbf{v} = \begin{pmatrix} \varpi (\tilde{g}_3p_{13}+\tilde{g}_1p_{14}) \\
\varpi (\tilde{g}_4p_{13}+\tilde{g}_2p_{14}) 
\\
0 \end{pmatrix}, p_{13},p_{14}\in \cO_2\right\rbrace,
$$
where $\mathbf{v}^T$ is the transpose of $\mathbf{v}$.
Let 
{\small{\begin{equation}\label{matriX C}
C=\begin{pmatrix}
    c_0 & c_2a & c_1a\\
    c_1 & c_0+c_2 & c_1+c_2a\\
    c_2 & c_1 & c_2+c_0
\end{pmatrix}\in \F_{q^3}\smallsetminus \F_q,
\end{equation}}}
where embedding of $\F_{q^3}$ in $M_3(\F_q)$ is defined in \eqref{Embedding of Fq^3}.
We have, 
\begin{equation}\label{computation of phi_B,1,chi-1}
(\tilde{\phi}_C\otimes 1\otimes \chi)(Y)=\psi_0((c_1g_3+c_2g_4)\bar{p}_{13})\psi_0((c_1g_1+c_2g_2)\bar{p}_{14}).
\end{equation}
For $\beta=\beta_1(x) \in N\backslash P/P_{(24)}$, we have
\begin{equation}\label{computation for coset a-1}
\psi^{\beta^{-1}(24)\gamma^{-1}}(Y)
=\psi_0(\bar{p}_{13}+\bar{x}\bar{p}_{14}).
\end{equation}
For fixed $\beta=\beta_1(x)$, $W_{\gamma, \beta}\neq 0$ if and only if there exists $g\in \GL_2(\F_q)$, such that the character values in \eqref{computation of phi_B,1,chi-1} and \eqref{computation for coset a-1} are the same,
equivalently
$$
c_1g_3+c_2g_4=1, \hspace{0.6 cm}c_1g_1+c_2g_2=\bar{x}.
$$
For any $x\in \cO_2$, we have $q(q-1)$ choices of $g$. 
Therefore, 
{\small{$$
\underset{\beta \in \{\beta_1(x):x\in \cO_2\}}{\sum}~\underset{\gamma\in L_0\backslash Q/S\cdot U\cdot \cO_2^\times}{\sum}
\dim(W_{\gamma,\beta})=q^2\cdot q(q-1).
$$ 
}}
For $\beta=
\beta_2(y)$, we have  
\begin{equation}\label{Computation for b coset-1}
\psi^{\beta^{-1}(24)\gamma^{-1}}(Y)=\psi_0(\bar{p}_{14}).
\end{equation}
Using the similar computations as for $\beta=\beta_1(x)$, for any $\beta=\beta_2(y)$, we get $q(q-1)$ choices of $g$ for which $W_{\gamma, \beta}\neq 0$.
Therefore, 
{\small{$$
\underset{\beta \in \{\beta_2(y):y\in \cO_2\}}{\sum}~\underset{\gamma\in L_0\backslash Q/S\cdot U\cdot \cO_2^\times}{\sum}
\dim(W_{\gamma,\beta})
=q\cdot q(q-1).
$$ 
}}
Hence, using \eqref{dimension of pi24}, we get
$\dim(\pi^{(24)}_{N,\psi})=q^3(q-1)+q^2(q-1) = q^2(q^2-1)$.
\qedhere
\end{proof}

The following lemma is a consequence of Mackey theory and Frobenius reciprocity, and we skip the proof.
\begin{lemma}\label{technical lemma for section-5-1}
Let $\sigma=\Ind_{I(\phi_B)}^{\GL_2(\cO_2)}\tilde{\phi}_B \in \Irr(\GL_2(\cO_2))$ be regular.
Then 
\begin{equation}\label{Equation for technical lemma 1}
\Hom_{\Delta\GL_2(\cO_2)\cdot N}(\pi^{(24)},\sigma\otimes\psi) 
\cong  \underset{\lambda\in \Lambda}{\bigoplus}W_{\lambda},
\end{equation}
where $\Lambda$ is a set of representatives for $\Delta I(\phi_B)\cdot N\backslash P/P_{(24)}$, $W_{\lambda}:=\Hom_{L_{\lambda}} \left(\rho\otimes \chi, (\tilde{\phi}_B\cdot \psi)^{\lambda^{-1}(24)} \right)$ and $L_{\lambda} = (24)[P_{(24)}\cap \lambda^{-1}\left(\Delta I(\phi_B)\cdot N\right)\lambda](24)$.
\end{lemma}

Recall the sets $\mathcal{X}_i$ for $i=1,2,3$ as defined in Section 2.3.

\begin{lemma}\label{4.13 lemma for coset representatives of gamma}
\begin{enumerate}[label =(\alph*)]
\item If $B \in \mathcal{X}_1$, 
then $\Lambda=\{I_4\}$.
\item If $B\in \mathcal{X}_2$, then $\Lambda$ can be taken to be {\small{$\left\lbrace \Diag( h , I_2):h\in \left\{I_2,w_0
\right\}\right\rbrace$}}.
\item If $B\in \mathcal{X}_3$, then $\Lambda$ can be taken to be
{\small{$\left\lbrace \Diag( h , I_2) :h\in \left\{I_2,
w_0, u^{-}(1)\right\}\right\rbrace$}}.
\end{enumerate}
\end{lemma}

\begin{proof}
Recall that $P_{(24)}=(24) Q(24)\cap P$.
We have a bijection between $N\backslash P/P_{(24)}$ and $\GL_2(\cO_2)/\mathfrak{B}_2$.
Thus we get a bijection between $\Lambda$ and $\bar{I}(\phi_B)\backslash \GL_2(\F_q)/\mathfrak{\bar{B}}_2$, and the lemma follows. 
\qedhere
\end{proof}

For $\lambda \in \Lambda$, the following series of lemmas explicitly describes the subgroup $L_{\lambda}= (24)[P_{(24)}\cap \lambda^{-1}\left(\Delta I(\phi_B)\cdot N\right)\lambda](24)$ for various $B$'s.

\begin{lemma}\label{subgroup L_I for non-split semisimple} 
Let $B\in \mathcal{X}_1$. 
If $\lambda=I_4$, then 
{\small{$L_{\lambda}=
\left\{\begin{pmatrix}
    p_{11} & p_{14} & p_{13} & \varpi p_{12}\\
    0 & p_{22} & 0 & 0\\
    0 & \varpi p_{12} & p_{11} & 0\\
    0 & 0 & 0 & p_{22}
\end{pmatrix}:\begin{matrix}
p_{11}, p_{22}\in \cO_2^\times,\\
p_{22}-p_{11}\in \varpi \cO_2,\\
p_{12}, p_{13},p_{14}\in \cO_2
\end{matrix}
\right\}$.}}
\end{lemma}

\begin{lemma}\label{subgroup L_beta for B unipotent}
Let $B\in \mathcal{X}_2$.
\begin{enumerate}[label = (\alph*)]
\item  If $\lambda=I_4$, then {\small{$L_{\lambda}=
\left\{\begin{pmatrix}
    p_{11} & p_{14} & p_{13} &  p_{12}\\
    0 & p_{22} & 0 & 0\\
    0 &  p_{12} & p_{11} & 0\\
    0 & 0 & 0 & p_{22}
\end{pmatrix}:\begin{matrix}
p_{11}\in \cO_2^\times,\\
p_{22}-p_{11}\in \varpi \cO_2,\\
p_{ij}\in \cO_2~\text{for}~i\neq j
\end{matrix}
\right\}$}}.
\item If $\lambda=\Diag(w_0, I_2)$, then {\small{$L_{\lambda}=
\left\lbrace \begin{pmatrix}
 p_{11} & p_{14} & p_{13} & \varpi p_{12}\\
    0 & p_{11} & \varpi p_{12} & 0 \\
    0 & 0 & p_{22} & 0\\
    0 & 0 & 0 & p_{22}
   \end{pmatrix}:\begin{matrix}
       p_{11}, p_{22}\in \cO_2^\times,\\
       p_{22}-p_{11}\in \varpi \cO_2,\\
       p_{ij}\in \cO_2 \text{ for } i\neq j
   \end{matrix}
\right\rbrace$}}.
\end{enumerate}
\end{lemma}

\begin{lemma}\label{subgroup L_beta for B split semisimple} 
Let $B\in \mathcal{X}_3$.
\begin{enumerate}[label =(\alph*)]
\item  If $\lambda=I_4$, then {\small{$L_{\lambda}=\left\{
\begin{pmatrix}
    p_{11} & p_{14} & p_{13} & \varpi p_{12}\\
    0 & p_{22} & 0 & 0\\
    0 & \varpi p_{12} & p_{11} & 0\\
    0 & 0 & 0 & p_{22}
\end{pmatrix}:\begin{matrix}
p_{11},p_{22}\in \cO_2^\times,
\\
p_{ij}\in \cO_2~\text{for}~i\neq j
\end{matrix}
\right\}$}}.
\item If $\lambda=\Diag(w_0, I_2)$, then {\small{$L_{\lambda}=
\left\lbrace \begin{pmatrix}
 p_{11} & p_{14} & p_{13} & \varpi p_{12}\\
    0 & p_{11} & \varpi p_{12} & 0\\
    0 & 0 & p_{22} & 0\\
    0 & 0 & 0 & p_{22}
   \end{pmatrix}:\begin{matrix}
       p_{11}\in \cO_2^\times,\\
       p_{22}-p_{11}\in \varpi \cO_2,\\
       p_{ij}\in \cO_2 \text{ for } i\neq j
   \end{matrix}
\right\rbrace$}}.
\item If $\lambda=\Diag(
    u^-(1) , I_2)$ , then 
{\small{$$
L_{\lambda}=
\left\lbrace \begin{pmatrix}
     p_{11} & p_{14} & p_{13} & \varpi p_{12}\\
    0 & p_{22}+\varpi p_{12} & p_{11}-p_{22}-\varpi p_{12} & 0\\
    0 & \varpi p_{12} & p_{11}-\varpi p_{12} & 0\\
    0 & 0 & 0 & p_{22}
\end{pmatrix} :\begin{matrix}
p_{22}-p_{11}\in \varpi \cO_2, \\
p_{11}\in \cO_2^\times, p_{ij}\in \cO_2 \text{ for $i\neq j$}
\end{matrix}
\right\rbrace.
$$}}
\end{enumerate}
\end{lemma}

The following lemma follows from Mackey theory and Lemma~\ref{technical lemma for section-5-1}.
\begin{lemma} \label{Hom_H_lambda}
For $\lambda \in \Lambda$, let $\Gamma_{\lambda}$ be a set of representatives for $L_{\lambda}\backslash Q/(S\cdot U\cdot \cO_2^\times)$.
Then
\begin{align} 
W_{\lambda} 
\cong 
\underset{\gamma\in \Gamma_{\lambda}}{\bigoplus} W_{\gamma, \lambda},
\end{align}
where $W_{\gamma, \lambda}:=\Hom_{\gamma^{-1}L_{\lambda}\gamma \cap (S\cdot U\cdot \cO_2^\times)}\left(\tilde{\phi}_{C}\otimes 1\otimes \chi, (\tilde{\phi}_B\cdot \psi)^{\lambda^{-1}(24)\gamma^{-1}}\right)$.
\end{lemma}

\begin{lemma}\label{reprsentatives for set Delta_beta}
For $B\in \mathcal{X}_1\bigcup \mathcal{X}_2\bigcup \mathcal{X}_3$ and $\lambda \in \Lambda$, there is a bijection between $\Gamma_{\lambda}$ and
$$
\{\gamma(g):=\Diag(1, g, 1)\in \GL_4(\F_q)~|~g\in \GL_2(\F_q)\}.
$$

\end{lemma}

\begin{proof}
Recall $\Gamma_{\lambda}$ from the above lemma.
It is clear that there is a bijection between 
\[ 
L_{\lambda} \backslash Q / (S \cdot U \cdot \cO_2^\times) \quad \text{and} \quad \bar{L}_{\lambda} \backslash \bar{Q} / (\F_{q^3}^\times \cdot \bar{U} \cdot \F_q^\times).
\]
Now, we use the description of $L_{\lambda}$ for each $B$ given in Lemma \ref{subgroup L_I for non-split semisimple}, \ref{subgroup L_beta for B unipotent} and \ref{subgroup L_beta for B split semisimple}.
The lemma follows from the decomposition
\(
\GL_3(\F_q) = \F_{q^3}^\times \cdot P_{1,2}
\), see Lemma \ref{double cosets P-1-n-2_GL_n-F_{q^n}}.
\end{proof}

\begin{proposition}\label{reduction to one gamma for split semisimple and unipotent-1}
\begin{enumerate}[label = (\alph*)]
\item Let $B\in \mathcal{X}_2$ and $\lambda=I_4$.
Then
$\underset{\gamma \in \Gamma_{\lambda}}{\bigoplus}W_{\gamma, \lambda} =0$.
\item Let $B\in \mathcal{X}_3 $ and $\lambda\in \Lambda$ such that $\lambda \neq \Diag( u^-(1), I_2)$.
Then 
$\underset{\gamma \in \Gamma_{\lambda}}{\bigoplus}W_{\gamma, \lambda} =0$.
\end{enumerate}
\end{proposition}

\begin{proof}
\begin{enumerate}[label = ($\alph*$)]
\item 
Let $B=\begin{pmatrix}
    m & 1\\
    0 & m
\end{pmatrix}\in \mathcal{X}_2$, $\lambda=I_4$ and $\gamma\in \Gamma_{\lambda}$.
% \begin{equation}\label{First claim of propsoition 4.13}
% W_{\gamma,\lambda} =0 \text{ for $\lambda=I$.}
% \end{equation}
% Using Lemma \ref{4.13 lemma for coset representatives of gamma} and assuming that \eqref{First claim of propsoition 4.13} holds, we prove part $(a)$ of the proposition.
% Now, we give a proof of \eqref{First claim of propsoition 4.13}.
% Let
% $B \in \mathcal{X}_2 \bigcup  \mathcal{X}_3$ and 
% let $\lambda = I$.
We use the description of $L_{\lambda}$ as given in Lemma~\ref{subgroup L_beta for B unipotent}.
Now, using Lemma \ref{reprsentatives for set Delta_beta}, for $\gamma = \gamma(g) $ with 
$ g = \begin{pmatrix}
g_1 & g_2 \\
g_3 & g_4
\end{pmatrix} $,
the subgroup $\gamma^{-1}L_{\lambda}\gamma\cap S\cdot  U\cdot \cO_2^\times$ is given by
$$
\left\lbrace
Y : =\begin{pmatrix}
p_{11} & \varpi \tilde{g}_3p_{13}+\varpi \tilde{g}_1p_{14} & \varpi \tilde{g}_2 p_{14}+\varpi \tilde{g}_4 p_{13} & \varpi p_{12}\\
0 & \frac{\tilde{g}_2\tilde{g}_3p_{11}-\tilde{g}_1\tilde{g}_4p_{22}+\varpi \tilde{g}_1\tilde{g}_2p_{12}}{\tilde{g}_2\tilde{g}_3-\tilde{g}_1\tilde{g}_4} & \frac{-\tilde{g}_2\tilde{g}_4(p_{22}-p_{11})+\varpi \tilde{g}_2^2p_{12}}{\tilde{g}_2\tilde{g}_3-\tilde{g}_1\tilde{g}_4} & 0\\
0 & \frac{\tilde{g}_1\tilde{g}_3(p_{22}-p_{11})-\varpi \tilde{g}_1^2p_{12}}{\tilde{g}_2\tilde{g}_3-\tilde{g}_1\tilde{g}_4} & \frac{\tilde{g}_2\tilde{g}_3p_{22}-\tilde{g}_1\tilde{g}_4p_{11}-\varpi \tilde{g}_1\tilde{g}_2p_{12}}{\tilde{g}_2\tilde{g}_3-\tilde{g}_1\tilde{g}_4} & 0\\
0 &  0 &  0 & p_{22}
\end{pmatrix}
~\middle | ~ \begin{matrix}
 p_{22}-p_{11}\in \varpi \cO_2,\\
 p_{11}\in \cO_2^\times,
 \\
 p_{ij}\in \cO_2\text{ for } i\neq j
\end{matrix}
 \right\rbrace.
$$ 
Write $p_{22}-p_{11} = \varpi t\in \varpi \cO_2$. 
Then
\begin{equation}\label{action of character in RHS for first gamma of split non-semisimple matrix-1}
(\tilde{\phi}_B\cdot \psi)^{(24)\gamma^{-1}}(Y)=\omega_{\sigma}(p_{11})\psi_0(m\bar{t})\psi_0(\bar{p}_{13}).
\end{equation}
Recall the matrix $C\in \F_{q^3}\smallsetminus \F_q$, as given in \eqref{matriX C}. 
Then
\begin{align}\label{action of character in LHS for first gamma of split non-semisimple matrix-1}
(\tilde{\phi}_C\otimes 1\otimes \chi)(Y) & =\omega_{\rho}(p_{11})\cdot \chi(p_{22})\cdot\psi_0\left((c_1g_1+c_2g_2)\bar{p}_{14}+(c_2g_4+c_1g_3)\bar{p}_{13}\right) \cdot 
\nonumber
\\
& \psi_0\left(\left((c_0+c_2)+\frac{g_1g_3(c_1+c_2a)-c_1g_2g_4}{(g_2g_3-g_1g_4)}\right)\bar{t}+\left(\frac{c_1g_2^2-(c_1+c_2a)g_1^2}{g_2g_3-g_1g_4}\right)\bar{p}_{12}\right).
\end{align}
Using \eqref{action of character in RHS for first gamma of split non-semisimple matrix-1} and \eqref{action of character in LHS for first gamma of split non-semisimple matrix-1}, $W_{\gamma, \lambda}\neq 0$ if and only if $\exists~g\in \GL_2(\F_q)$ such that
\begin{align*}
\psi_0\left((c_1g_3+c_2g_4)\bar{p}_{13}\right)=\psi_0(\bar{p}_{13}),~\psi_0((c_1g_1+c_2g_2)\bar{p}_{14})=1,
\psi_0\left((c_1g_2^2-(c_1+c_2a)g_1^2)\bar{p}_{12}\right)=1,
\end{align*}
which gives 
\begin{equation}\label{3 equations -2-1}
c_1g_3+c_2g_4=1, \hspace{0.6 cm} c_1g_1+c_2g_2=0, \hspace{0.6 cm} c_1g_2^2-(c_1+c_2a)g_1^2=0.
\end{equation}
This implies $c_2 \neq 0$ and $c_1^3 - c_1 c_2^2 - c_2^3a = 0$, a contradiction to the irreducibility of the polynomial $x^3-x-a$ (see Section \ref{all embeddings}).
This proves part $(a)$.
\item
Let $B=\diag(m_1,n_1)\in \mathcal{X}_3$, $\lambda=I_4$ and $\gamma\in \Gamma_{\lambda}$. 
Then $\gamma^{-1}L_{\lambda}\gamma \cap S\cdot U\cdot \cO_2^\times$ is same as that in part $(a)$. 
Then $W_{\gamma,\lambda}=0$ follows along similar line.
Now, by Lemma~\ref{4.13 lemma for coset representatives of gamma}, it is enough to prove $W_{\gamma, \lambda}=0$ for 
$\lambda=\Diag(w_0, I_2)$
and $ \gamma \in \Gamma_{\lambda}$.
Using Lemma \ref{reprsentatives for set Delta_beta}, for $\gamma = \gamma(g)$ with 
$g = \begin{pmatrix}
g_1 & g_2 \\
g_3 & g_4
\end{pmatrix}$,
the subgroup $\gamma^{-1}L_{\lambda}\gamma\cap S\cdot U\cdot \cO_2^\times$ is given by
$$
\left\lbrace
Y':=\begin{pmatrix}
p_{11} & \varpi \tilde{g}_3p_{13}+\varpi \tilde{g}_1p_{14} & \varpi \tilde{g}_2p_{14}+\varpi \tilde{g}_4p_{13}
& \varpi p_{12}\\
0 & \frac{\tilde{g}_2\tilde{g}_3p_{22}-\tilde{g}_1\tilde{g}_4p_{11}-\varpi \tilde{g}_3\tilde{g}_4p_{12}}{\tilde{g}_2\tilde{g}_3-\tilde{g}_1\tilde{g}_4} & \frac{\tilde{g}_2\tilde{g}_4(p_{22}-p_{11})-\varpi \tilde{g}_4^2p_{12}}{\tilde{g}_2\tilde{g}_3-\tilde{g}_1\tilde{g}_4} & 0\\
0 & \frac{-\tilde{g}_1\tilde{g}_3(p_{22}-p_{11})+\varpi \tilde{g}_3^2p_{12}}{\tilde{g}_2\tilde{g}_3-\tilde{g}_1\tilde{g}_4} & \frac{\tilde{g}_2\tilde{g}_3p_{11}-\tilde{g}_1\tilde{g}_4p_{22}+\varpi \tilde{g}_3\tilde{g}_4p_{12}}{\tilde{g}_2\tilde{g}_3-\tilde{g}_1\tilde{g}_4} & 0\\
0 &  0 &  0 & p_{22}
\end{pmatrix}
 ~\middle |~\begin{matrix}
 p_{22}-p_{11}\in \varpi \cO_2,\\
 p_{11}\in \cO_2^\times,
 \\
 p_{ij}\in \cO_2\text{ for } i\neq j
\end{matrix}
\right\rbrace.
$$
Write $p_{22}-p_{11}=\varpi t\in \varpi \cO_2$.
Then
\begin{equation}\label{action of character in RHS for first gamma of split non-semisimple matrix}
(\tilde{\phi}_B\cdot \psi)^{(24)\gamma^{-1}}(Y')=\omega_{\sigma}(p_{11})\psi_0(m_1\bar{t})\psi_0(\bar{p}_{14})
\end{equation}
and 
{\small{\begin{align}\label{action of character in LHS for first gamma of split non-semisimple matrix}
(\tilde{\phi}_C\otimes 1\otimes \chi)(Y')= 
& \omega_{\rho}(p_{11})\cdot \chi(p_{22})\cdot\psi_0\left((c_1g_1+c_2g_2)\bar{p}_{14}+(c_2g_4+c_1g_3)\bar{p}_{13}\right)\nonumber 
\\
& \cdot 
\psi_0\left(\left(c_0+c_2+\frac{c_1g_2g_4-g_1g_3(c_1+c_2a)}{(g_2g_3-g_1g_4)}\right)\bar{t}+ \left (\frac{(c_1+c_2a)g_3^2-c_1g_4^2}{g_2g_3-g_1g_4}\right)\bar{p}_{12}\right).
\end{align}}}
Thus $W_{\gamma, \lambda}\neq 0$
if and only if there exists $g\in \GL_2(\F_q)$ such that
\begin{equation}\label{3 equations -2}
c_1g_3+c_2g_4=0, \hspace{0.6 cm} c_1g_1+c_2g_2=1,\hspace{0.6 cm} c_1g_4^2-(c_1+c_2a)g_3^2=0.
\end{equation}
From \eqref{3 equations -2}, we get a contradiction as in \eqref{3 equations -2-1}.
Thus, the proposition follows.
\qedhere
\end{enumerate}
\end{proof}

\begin{proposition}\label{First proposition for non-split semisimple with central character involving m_0}
Let $m_0 \in \F_q$ be such that $\psi_0(2m_0 x) = \omega_{\pi}(1+ \varpi \tilde{x})$ for all $x \in \F_q$, where $\tilde{x} \in \cO_2$ is a lift of $x$. 
Let $B \in \mathcal{X}^0_1 := \{ A \in \mathcal{X}_1 : \tr(A) = 2m_0 \}$.
%with $\tr(B)=2m_0$. 
Let $\tilde{\phi}_{B}$ be an extension of $\phi_{B}$ to $I(\phi_B)$  and $\sigma = \Ind_{I(\phi_B)}^{\GL_2(\cO_2)} \tilde{\phi}_{B}$. 
Then
$m(\pi^{(24)}_{N,\psi},\sigma)=1$.
\end{proposition} 

\begin{proof}
Write $B=\begin{pmatrix}
    m_0 & n\alpha\\
    n & m_0
\end{pmatrix}$ with $n\in \F_q^\times$.
Using Lemma \ref{technical lemma for section-5-1}, Lemma \ref{4.13 lemma for coset representatives of gamma} and Lemma \ref{Hom_H_lambda}, we have
\begin{equation}\label{Equation 4.66}
\Hom_{\Delta \GL_2(\cO_2)\cdot N}(\pi^{(24)},\sigma\otimes\psi)
\cong \underset{\gamma \in \Gamma_{I}}{\bigoplus}W_{\gamma, I}.
\end{equation}
Using $L_I$ from Lemma \ref{subgroup L_I for non-split semisimple} and $\Gamma_{I}$ from Lemma \ref{reprsentatives for set Delta_beta}, if $\gamma=\gamma(g)$ for $g=\begin{pmatrix}
     g_1 & g_2\\
     g_3 & g_4
 \end{pmatrix}$ then subgroup $\gamma^{-1}L_{I}\gamma \cap S\cdot U\cdot \cO_2^\times$ is given by
$$
{\small{\left\lbrace
Y: =\begin{pmatrix}
p_{11} & \varpi \tilde{g}_3p_{13}+\varpi \tilde{g}_1p_{14} & \varpi \tilde{g}_4p_{13}+\varpi \tilde{g}_2p_{14} & \varpi p_{12}
\\
0 & p_{11}+\frac{\varpi \tilde{g}_1\tilde{g}_2p_{12}-\tilde{g}_1\tilde{g}_4(p_{22}-p_{11})}{\tilde{g}_2\tilde{g}_3-\tilde{g}_1\tilde{g}_4}	& \frac{-\tilde{g}_2\tilde{g}_4(p_{22}-p_{11})+\varpi \tilde{g}_2^2p_{12}}{\tilde{g}_2\tilde{g}_3-\tilde{g}_1\tilde{g}_4}	& 0\\
0 &  \frac{\tilde{g}_1\tilde{g}_3(p_{22}-p_{11})-\varpi \tilde{g}_1^2p_{12}}{\tilde{g}_2\tilde{g}_3-\tilde{g}_1\tilde{g}_4}	& p_{11}+\frac{\tilde{g}_2\tilde{g}_3(p_{22}-p_{11})-\varpi \tilde{g}_1\tilde{g}_2p_{12}}{\tilde{g}_2\tilde{g}_3-\tilde{g}_1\tilde{g}_4}	 & 0\\
0 &  0 &  0	 &  p_{22}
\end{pmatrix} ~\middle|~
\begin{matrix}
p_{22}-p_{11}\in \varpi \cO_2,\\
p_{11}\in \cO_2^\times, 
\\
p_{ij}\in \cO_2 \text{ for $i\neq j$}
\end{matrix}
\right\rbrace
}}.
$$
Let $p_{22}-p_{11}=\varpi t\in \varpi \cO_2$, we have
\begin{equation}\label{action of character in LHS of strongly cuspidal}
(\tilde{\phi}_B\cdot \psi)^{(24)\gamma^{-1}}(Y)=\omega_{\sigma}(p_{11})\tilde{\phi}_B\begin{pmatrix}
    1 & \varpi p_{12}\\
    0 & \varpi t
\end{pmatrix}\psi_0(\varpi p_{13})= \omega_{\sigma}(p_{11})\psi_0\left(m_0\bar{t}+n\bar{p}_{12}\right)\psi_0(\bar{p}_{13}).
\end{equation}
Recall the matrix $C$ as given in \eqref{matriX C}, we have
% The character $\tilde{\phi}_C\otimes 1\otimes \chi$ at $Y$ is given by
{\small{\begin{align}
\label{action of character in RHS of strongly cuspidal}
(\tilde{\phi}_C\otimes 1\otimes \chi)(Y) = & \omega_{\pi}(p_{11})\cdot \chi(1+\varpi t)\cdot \psi_0\left((c_1g_3+c_2g_4)\bar{p}_{13}+(c_2g_2+c_1g_1)\bar{p}_{14}\right)
\nonumber
\\
&
\cdot \psi_0\left(
\left(c_0+c_2+\frac{(c_1+c_2a)g_1g_3-c_1g_2g_4}{g_2g_3-g_1g_4}\right)\bar{t}+\left(\frac{c_1g_2^2-(c_1+c_2a)g_1^2}{(g_2g_3-g_1g_4)}\right)\bar{p}_{12}\right).
\end{align}
}}
Using \eqref{action of character in LHS of strongly cuspidal} and  \eqref{action of character in RHS of strongly cuspidal}, we get that $W_{\gamma, I} \neq 0$ if and only if
\begin{enumerate}[label = (\roman*)]
\item $\omega_{\pi}(1+\varpi \tilde{x})=\psi_0(2m_0x)$, which is an assumption,
\item  $c_1g_1+c_2g_2=0$,        \hspace{0.3 cm } $c_1g_3+c_2g_4=1$,  \hspace{0.3 cm } $\dfrac{c_1g_2^2-(c_1+c_2a)g_1^2}{g_2g_3-g_1g_4}=n$ and 
\\
$\psi_0\left(
\left(c_0+c_2+\dfrac{(c_1+c_2a)g_1g_3-c_1g_2g_4}{g_2g_3-g_1g_4}\right)\bar{t}\right)
\chi(1+\varpi t)=\psi_0(m_0\bar{t})$.
\end{enumerate}
The condition (ii) gives a unique $g\in \GL_2(\F_q)$ and hence a unique $\gamma \in \Gamma_{I}$ for which $W_{\gamma, I} \neq 0$.
Thus, using Equation \eqref{Equation 4.66}, the proposition follows.
\end{proof}

\begin{corollary}\label{corollary 1}
Let $m_0$ and $B$ be as in the previous proposition.
Then
\( \underset{B \in \mathcal{X}_1^0 / \sim}{\bigoplus} \Ind_{Z \cdot J^1_2}^{\GL_2(\cO_2)} (\omega_{\pi} \phi_B) \subseteq \pi^{(24)}_{N,\psi},
\)
where $A_1 \sim A_2 \iff A_1$ and  $A_2$ are similar matrices.
\end{corollary}

\begin{proof}
For $B \in \mathcal{X}_1^0$, the representation $\Ind_{Z \cdot J^1_2}^{I(\phi_B)} (\omega_{\pi} \phi_{B})$ is a direct sum of characters of $I(\phi_{B})$ that extend the character $\omega_{\pi} \phi_{B}$ of $Z\cdot J^1_2$. 
Moreover, if $\tilde{\phi}_{B}$ is such an extension of $\phi_{B}$ and $\sigma=\Ind_{I(\phi_B)}^{\GL_2(\cO_2)}\tilde{\phi}_B$, then $m(\pi^{(24)}_{N, \psi}, \sigma)=1$, i.e. $\sigma \subseteq \pi^{(24)}_{N, \psi}$. 
Now the corollary follows by using 
\[
\Ind_{Z\cdot J^1_2}^{\GL_2(\cO_2)} ( \omega_{\pi}\phi_B ) = \Ind_{I(\phi_B)}^{\GL_2(\cO_2)} \left( \Ind_{Z \cdot J^1_2}^{I(\phi_B)} ( \omega_{\pi}\phi_B) \right) \subseteq \pi^{(24)}_{N,\psi}.
\qedhere
\]
\end{proof}

\begin{proposition}
\label{Third proposition for Q}
Let $m_0 \in \F_q$ be such that $\psi_0(2m_0 x) = \omega_{\pi}(1+ \varpi \tilde{x})$ for all $x \in \F_q$, where $\tilde{x} \in \cO_2$ is any lift of $x$.
Let $B \in \mathcal{X}^0_2:=\{ A \in \mathcal{X}_2: \tr(A)=2m_0\}$.
Let $\tilde{\phi}_{B}$ be an extension to $I(\phi_B)$ of the character $\phi_{B}$ of $J^1_2$ and $\sigma = \Ind_{I(\phi_B)}^{\GL_2(\cO_2)} \tilde{\phi}_{B}$.
Then  $m(\pi^{(24)}_{N,\psi}, \sigma)=1$.
\end{proposition}

\begin{proof}
Write $B=\begin{pmatrix}
    m_0 & 1\\
    0 & m_0
\end{pmatrix}\in \mathcal{X}^0_2$.
Using Lemma \ref{technical lemma for section-5-1}, Lemma \ref{Hom_H_lambda} and Proposition \ref{reduction to one gamma for split semisimple and unipotent-1} $(a)$, we have
\begin{equation}\label{4.57}
\Hom_{\Delta \GL_2(\cO_2) \cdot N}(\pi^{(24)},\sigma\otimes\psi)
\cong 
W_{\gamma,\lambda}
\end{equation}
where $\lambda=\Diag(w_0, I_2)$.
Using $L_{\lambda}$ from Lemma \ref{subgroup L_beta for B unipotent} and $\Gamma_{\lambda}$ from Lemma \ref{reprsentatives for set Delta_beta}, for $\gamma = \gamma (g)$ where $g=\begin{pmatrix}
     g_1 & g_2\\
     g_3 & g_4
 \end{pmatrix}$,  the subgroup $\gamma^{-1}L_{\lambda}\gamma \cap S\cdot U\cdot \cO_2^\times$ is given by 
$$
\left\lbrace
Y := \begin{pmatrix}
p_{11} & \varpi \tilde{g}_3p_{13}+\varpi \tilde{g}_1p_{14} & \varpi \tilde{g}_2p_{14}+\varpi \tilde{g}_4p_{13} & \varpi p_{12}\\
0 & \frac{\tilde{g}_2\tilde{g}_3p_{22}-\tilde{g}_1\tilde{g}_4p_{11}-\varpi \tilde{g}_3\tilde{g}_4p_{12}}{\tilde{g}_2\tilde{g}_3-\tilde{g}_1\tilde{g}_4} & \frac{\tilde{g}_2\tilde{g}_4(p_{22}-p_{11})-\varpi \tilde{g}_4^2p_{12}}{\tilde{g}_2\tilde{g}_3-\tilde{g}_1\tilde{g}_4} & 0\\
0 & \frac{-\tilde{g}_1\tilde{g}_3(p_{22}-p_{11})+\varpi \tilde{g}_3^2p_{12}}{\tilde{g}_2\tilde{g}_3-\tilde{g}_1\tilde{g}_4} & \frac{\tilde{g}_2\tilde{g}_3p_{11}-\tilde{g}_1\tilde{g}_4p_{22}+\varpi \tilde{g}_3\tilde{g}_4p_{12}}{\tilde{g}_2\tilde{g}_3-\tilde{g}_1\tilde{g}_4} & 0\\
0 &  0 &  0 & p_{22}
\end{pmatrix}
~\middle |~
\begin{matrix}
     p_{22}-p_{11}\in \varpi \cO_2,\\
     p_{11}\in \cO_2^\times,
     \\
     p_{ij}\in \cO_2 \text{ for } i\neq j
\end{matrix}\right\rbrace.
$$
Write $p_{22}-p_{11}=\varpi t\in \varpi \cO_2$. 
Then
\begin{equation}\label{action of character of LHS of last gamma for split non-semisimple}
(\tilde{\phi}_B\cdot \psi)^{\lambda^{-1}(24)\gamma^{-1}}(Y)=\omega_{\sigma}(p_{11})\psi_0(m_0\bar{t})\psi_0(\bar{p}_{14})\psi_0(\bar{p}_{12}).
\end{equation}
Recall the matrix $C$, as given in \eqref{matriX C}. 
Then
{\small{\begin{align}\label{action of character of RHS of last gamma for split non-semisimple}
(\tilde{\phi}_C\otimes 1\otimes \chi)(Y)= 
& \omega_{\pi}(p_{11})\cdot \chi(1+\varpi t)\cdot \psi_0\left(\varpi (c_1g_3+c_2g_4)\bar{p}_{13}+\varpi (c_1g_1+c_2g_2)\bar{p}_{14}\right)
\nonumber
\\
& 
\cdot \psi_0\left(\left(c_0+c_2+\frac{c_1g_2g_4-(c_1+c_2a)g_1g_3}{g_2g_3-g_1g_4}\right)\bar{t}
+\left(\frac{ (c_1+c_2a)g_3^2-c_1g_4^2}{g_2g_3-g_1g_4}\right)\bar{p}_{12}\right).
\end{align}
}}
Using \eqref{action of character of LHS of last gamma for split non-semisimple} and \eqref{action of character of RHS of last gamma for split non-semisimple}, 
$W_{\gamma, \lambda}\neq 0$ if and only if 
\begin{enumerate}[label = (\roman*)]
\item $\omega_{\pi}(1+\varpi \tilde{x})=\psi_0(2m_0x)$ which is an assumption,
\item $c_1g_3+c_2g_4=0$, \hspace{0.3 cm}
$c_1g_1+c_2g_2=1$, \hspace{0.3 cm}
$\dfrac{(c_1+c_2a)g_3^2-c_1g_4^2}{g_2g_3-g_1g_4}=1$ and 
\\
$\psi_0\left(\left(c_0+c_2+\dfrac{c_1g_2g_4-(c_1+c_2a)g_1g_3}{g_2g_3-g_1g_4}\right)
\bar{t}\right)\chi(1+\varpi t)=\psi_0(m_0\bar{t})~\forall~t\in \cO_2$.
\end{enumerate}
The condition (ii) gives a unique $g\in \GL_2(\F_q)$ and hence a unique $\gamma \in \Gamma_{\lambda}$ for which $W_{\gamma, \lambda} \neq 0$.
Thus, using \eqref{4.57}, the proposition follows.
\qedhere
\end{proof}

\begin{corollary}
\label{corollary 3}
Let $m_0$ and $B$ be as in the previous proposition.
Then
\(
\Ind_{Z \cdot J^1_2}^{\GL_2(\cO_2)} (\omega_{\pi} \phi_B) \subseteq \pi^{(24)}_{N,\psi}.
\)
\end{corollary}

\begin{proposition}\label{Second proposition for Q}
Let $m_0 \in \F_q$ be such that $\psi_0(2m_0 x) = \omega_{\pi}(1+ \varpi \tilde{x})$ for all $x \in \F_q$, where $\tilde{x} \in \cO_2$ is any lift of $x$.
Let $B \in \mathcal{X}^0_3 := \{ A \in \mathcal{X}_3 : \tr(A)=2m_0 \}$.
Let $\tilde{\phi}_{B}$ be an extension to $I(\phi_B)$ of the character $\phi_{B}$ of $J^1_2$ and $\sigma = \Ind_{I(\phi_B)}^{\GL_2(\cO_2)} \tilde{\phi}_{B}$.
Then  $m(\pi^{(24)}_{N,\psi}, \sigma)=1$.
\end{proposition}

\begin{proof}
Write $B=\diag(m,n)$ with $m\neq m_0$ and $n=2m_0-m$.
Using Lemma \ref{technical lemma for section-5-1}, Lemma \ref{Hom_H_lambda} and Proposition \ref{reduction to one gamma for split semisimple and unipotent-1} $(b)$,  we get
\begin{align}\label{For second induction and principal series-1}
\Hom_{\Delta \GL_2(\cO_2)\cdot N}(\pi^{(24)},\sigma\otimes\psi)
\cong \underset{\gamma\in \Gamma_{\lambda}}{\bigoplus}W_{\gamma, \lambda}
\end{align}
where $\lambda=\Diag(u^{-}(1), I_2)$.
Using $L_{\lambda}$ from Lemma \ref{subgroup L_beta for B split semisimple} and $\Gamma_{\lambda}$ from Lemma \ref{reprsentatives for set Delta_beta}, if $\gamma = \gamma(g)$ for $g=\begin{pmatrix}
     g_1 & g_2\\
     g_3 & g_4
 \end{pmatrix}$ then the subgroup $\gamma^{-1}L_{\lambda}\gamma \cap S\cdot U\cdot \cO_2^\times$ consists of matrices of the form
\begin{equation} \label{matrix Q_0}   
Y :  = 
\begin{pmatrix}
Y_1 & Y_2 \\
0 &  Y_3
\end{pmatrix}, 
\end{equation}
where 
$Y_2 : = 
\begin{pmatrix}
\varpi p_{12} & 0 & 0
\end{pmatrix}^{T}\in M_{3\times 1}(\cO_2),~
Y_3 : = 
\begin{pmatrix}
p_{22}
\end{pmatrix}\in \GL_1(\cO_2)$
and $ Y_1 : =\begin{pmatrix}
y_{ij}
\end{pmatrix}\in \GL_3(\cO_2) $ as described below.
Write $D := \tilde{g}_2 \tilde{g}_3 - \tilde{g}_1 \tilde{g}_4$,
{\small{\[
\begin{aligned}
y_{11} &= p_{11} \text{ such that $p_{22}-p_{11}\in \varpi \cO_2$,} \quad y_{12} = \tilde{g}_3 p_{13}+ \tilde{g}_1 p_{14}, \quad y_{13} = \tilde{g}_4 p_{13} + \tilde{g}_2 p_{14},
\\
y_{21} & =0, \quad y_{22} = \frac{(\tilde{g}_2 \tilde{g}_3 - \tilde{g}_3 \tilde{g}_4) p_{11} + (\tilde{g}_3 \tilde{g}_4 - \tilde{g}_1 \tilde{g}_4) p_{22} + \varpi (\tilde{g}_1 \tilde{g}_2 - \tilde{g}_2 \tilde{g}_3 + \tilde{g}_3 \tilde{g}_4 - \tilde{g}_1 \tilde{g}_4) p_{12}}{D}, \\
y_{23} & = \frac{\tilde{g}_4^2 - \tilde{g}_2 \tilde{g}_4 (p_{22} - p_{11}) + \varpi p_{12} (\tilde{g}_2 - \tilde{g}_4)^2}{D}, \quad y_{31}= 0, \quad y_{32} = \frac{\tilde{g}_3^2 - \tilde{g}_1 \tilde{g}_3 (p_{11} - p_{22}) - \varpi p_{12} (\tilde{g}_1 - \tilde{g}_3)^2}{D},
\\
y_{33} & = \frac{(\tilde{g}_2 \tilde{g}_3 - \tilde{g}_3 \tilde{g}_4) p_{11} + (\tilde{g}_3 \tilde{g}_4 - \tilde{g}_1 \tilde{g}_4) p_{22} - \varpi (\tilde{g}_1 \tilde{g}_2 - \tilde{g}_2 \tilde{g}_3 + \tilde{g}_3 \tilde{g}_4 - \tilde{g}_1 \tilde{g}_4) p_{12}}{D}.
\end{aligned}
\]
}}
Write $p_{22}-p_{11}=\varpi t\in \varpi\cO_2$.
Then 
\begin{equation}\label{action of character computation of LHS for last gamma for split semisimple-1}
(\tilde{\phi}_B\psi)^{\lambda^{-1}(24)\gamma^{-1}}(Y) =\omega_{\sigma}(p_{11})\psi_0(n\bar{t})\psi_0(\bar{p}_{14})\psi_0(\bar{p}_{13})\psi_0((n-m)\bar{p}_{12}).
\end{equation} 
Recall the matrix $C$ as given in \eqref{matriX C}. 
Then
{\small{\begin{align}\label{action of character computation of RHS for last gamma for split semisimple-1}
(\tilde{\phi}_C\otimes 1\otimes \chi) (Y) =
&
~\omega_{\pi}(p_{11})\cdot \chi(1+\varpi t)\cdot \psi_0\left((c_1g_3+c_2g_4)\bar{p}_{13}+(c_1g_1+c_2g_2)\bar{p}_{14}\right)
\nonumber
\\
& \cdot  \psi_0\left(\left((c_0+c_2)+\frac{c_1(g_4^2-g_2g_4)-(c_1+c_2a)(g_3^2-g_1g_3)}{(g_2g_3-g_1g_4)}\right)\bar{t}\right)
\nonumber
\\
& \cdot  \psi_0\left(\left(\frac{c_1(g_2-g_4)^2-(c_1+c_2a)(g_1-g_3)^2}{g_2g_3-g_1g_4}\right) \bar{p}_{12}\right).
\end{align}
}}
Now using \eqref{action of character computation of LHS for last gamma for split semisimple-1} and \eqref{action of character computation of RHS for last gamma for split semisimple-1} we get that $W_{\gamma, \lambda}\neq 0$ if and only if 
\begin{enumerate}[label =(\roman*)]
\item $\omega_{\pi}(1+\varpi \tilde{x})=\psi_0(2m_0x)$ which is an assumption,
\item  $c_1g_3+c_2g_4=1$, \hspace{0.3 cm}
$c_1g_1+c_2g_2=1$, \hspace{0.3 cm} $\frac{c_1(g_2-g_4)^2-(c_1+c_2a)(g_1-g_3)^2}{g_2g_3-g_1g_4}=n-m$ and
\\
$\psi_0\left(\left(c_0+c_2 +\frac{c_1(g_4^2-g_2g_4)-(c_1+c_2a)(g_3^2-g_1g_3)}{(g_2g_3-g_1g_4)}\right)
\bar{t}\right)\chi(1+\varpi t)=\psi_0( n\bar{t})$ for all  
$t\in \cO_2$.
\end{enumerate}
The condition (ii) gives a unique $g\in \GL_2(\F_q)$ and hence a unique $\gamma \in \Gamma_{\lambda}$ for which $W_{\gamma, \lambda} \neq 0$.
Therefore, using \eqref{For second induction and principal series-1}, the proposition follows.
\qedhere
\end{proof}

\begin{corollary}\label{corollary 2}
Let $m_0$ and $B$ be as in the previous proposition.
Then
\(
 \underset{B \in \mathcal{X}_3^0 / \sim}{\bigoplus} \Ind_{Z \cdot J^1_2}^{\GL_2(\cO_2)} (\omega_{\pi} \phi_B) \subseteq \pi^{(24)}_{N,\psi},
\)
where $A_1 \sim A_2 \iff A_1$ and  $A_2$ are similar matrices.
\end{corollary}

%Thus, from all the results discussed in this section, we conclude the following. 

% \begin{theorem} \label{induction from Q}
% Let $\pi = \Ind_Q^{\GL_4(\cO_2)}(\rho \otimes \chi)
% $, where $\rho$ is a strongly cuspidal representation of $\GL_3(\cO_2)$ and $\chi$ is a character of $\cO_2^\times$.
% Let $m_0 \in \F_q$ be such that 
% $\psi_0(2m_0 x) = \omega_{\pi}(1+ \varpi \tilde{x})$
% for all $x \in \F_q$, where $\tilde{x} \in \cO_2$ is any lift of $x$. 
% Then 
% \[
% \pi_{N,\psi}\cong \bigoplus_{B} \Ind_{Z \cdot J^1_2}^{\GL_2(\cO_2)} (\omega_{\pi} \phi_B),
% \]
% where $B\in (\mathcal{X}^0_1  \bigcup \mathcal{X}^0_2  \bigcup \mathcal{X}^0_3) / \sim$, where $A_1 \sim A_2 \iff A_1$ and  $A_2$ are similar matrices.
% \end{theorem}

Now we prove Theorem \ref{induction from Q-intro} mentioned in the introduction.
\begin{proof}[Proof of Theorem \ref{induction from Q-intro}]
By Proposition \ref{pi^I, pi_I_w have zero Whittaker models}, $\pi_{N, \psi} \cong \pi^{(24)}_{N, \psi}$.
By Corollary ~\ref{corollary 1}, Corollary \ref{corollary 3} and   Corollary \ref{corollary 2}  we have 
\[
\underset{B \in \mathcal{X}_i^0 / \sim}{\bigoplus} \Ind_{Z \cdot J^1_2}^{\GL_2(\cO_2)} (\omega_{\pi} \phi_B) \subseteq \pi^{(24)}_{N,\psi} ~~~~~~\text{ for } ~~~i=1,2,3.
\] 
Therefore 
\[
\underset{B \in (\mathcal{X}^0_1  \bigcup \mathcal{X}^0_2  \bigcup \mathcal{X}^0_3) / \sim}{\bigoplus} \Ind_{Z \cdot J^1_2}^{\GL_2(\cO_2)} (\omega_{\pi} \phi_B) \subseteq \pi^{(24)}_{N, \psi}.
\]
Observe that
\[
\underset{B \in (\mathcal{X}^0_1  \bigcup \mathcal{X}^0_2  \bigcup \mathcal{X}^0_3) / \sim}{\sum} \dim \left(\Ind_{Z \cdot J^1_2}^{\GL_2(\cO_2)} (\omega_{\pi} \phi_B) \right)
= q^2(q^2-1)
= \dim(\pi^{(24)}_{N,\psi}),
\]
where the second equality follows from Theorem \ref{dim of pi^24_N,psi}.
Note that $(\mathcal{X}^0_1  \bigcup \mathcal{X}^0_2  \bigcup \mathcal{X}^0_3) / \sim$ is a set of representative of conjugacy classes of regular element of $M_2(\F_q)$ with trace $2m_0$.
Hence $\pi_{N, \psi} \cong \bigoplus_{B} \Ind_{Z \cdot J^1_2}^{\GL_2(\cO_2)} (\omega_{\pi} \phi_B)$, where $B$ varies over the set of equivalence classes of all regular elements of $M_2(\F_q)$ with trace $2m_0$, 
proving Theorem \ref{induction from Q-intro}.
% The theorem follows from Corollary ~\ref{corollary 1},  Corollary \ref{corollary 2}, Corollary \ref{corollary 3} and using the fact that
% \[
% \underset{B}{\sum} \dim \left(\Ind_{Z \cdot J^1_2}^{\GL_2(\cO_2)} (\omega_{\pi} \phi_B) \right)=\dim(\pi^{(24)}_{N,\psi})=q^2(q^2-1).
% \qedhere
% \]
\end{proof}

The following corollaries are easy consequences of Theorem \ref{induction from Q-intro}.
\begin{corollary}
Let $\pi$ be as in Theorem \ref{induction from Q-intro}.
Then
\begin{enumerate}[label = ($\alph*$)]
\item The degenerate Whittaker space $\pi_{N, \psi}$ consists of all the regular representations of $\GL_2(\cO_2)$ with central character as $\omega_{\pi}$.
\item $\pi_{N, \psi}$ is a multiplicity-free representation. 
\end{enumerate}  
\end{corollary}
\begin{proof}
Part $(a)$ follows from Section \ref{reg reps of GL_2}.
Part $(b)$ follows from the fact that $\Ind_{Z \cdot J^1_2}^{\GL_2(\cO_2)} (\omega_{\pi} \phi_B)$ is multiplicity-free for a regular $B$. 
\end{proof}
\begin{corollary}
Let $\pi_1 = \Ind_Q^{\GL_4(\cO_2)}(\rho_1 \otimes \chi_1)$ and $\pi_2 = \Ind_Q^{\GL_4(\cO_2)}(\rho_2 \otimes \chi_2)$, where $\rho_1, \rho_2$ are strongly cuspidal representations of $\GL_3(\cO_2)$ and $\chi_1, \chi_2$ are characters of $\cO_2^{\times}$ (as in Theorem \ref{induction from Q-intro}). If the central characters of $\pi_1$ and $\pi_2$ are the same, then $(\pi_1)_{N, \psi} \cong (\pi_2)_{N, \psi}$ as representations of $\GL_2(\cO_2)$. 
\end{corollary}
\begin{proof}
This follows from the fact that $(\pi_1)_{N, \psi}$ and $(\pi_2)_{N, \psi}$ both consist of all regular representations of $\GL_2(\cO_2)$ with central characters $\omega_{\pi_1}$ and $\omega_{\pi_2}$, respectively, which are assumed to be the same.   
\end{proof}

\section*{Acknowledgements}
The authors thank Dipendra Prasad for several helpful conversations and insightful comments. 
AP acknowledges the support of PMRF fellowship (1400772). 
SPP acknowledges the support of the SERB MATRICS grant (MTR/2022/000782).

\bibliographystyle{alpha}

\end{document}